\documentclass[11pt,reqno]{amsart}
\usepackage{amsmath}
\usepackage{amstext}
\usepackage{amsbsy}
\usepackage{amsopn}
\usepackage{upref}
\usepackage{amsthm}
\usepackage{amsfonts}
\usepackage{amssymb}
\usepackage{mathrsfs}
\usepackage{times}
\allowdisplaybreaks
 \newtheorem{theorem}{Theorem}

 \newtheorem{lemma}[theorem]{Lemma}
\theoremstyle{definition}
 
\theoremstyle{remark}
 
\begin{document}
\title[Dispersive Systems]{Fourth-order dispersive systems on the one-dimensional torus}
\author[H.~Chihara]{Hiroyuki Chihara}
\address{Department of Mathematics and Computer Science, Kagoshima University, Kagoshima 890-0065, Japan}
\email{chihara@sda.att.ne.jp}
\thanks{Supported by JSPS Grant-in-Aid for Scientific Research \#23340033.}
\begin{abstract}
We present the necessary and sufficient conditions of 
the well-posedness of the initial value problem for 
certain fourth-order linear dispersive systems on the one-dimensional torus. 
This system is related with 
a dispersive flow for closed curves into compact Riemann surfaces. 
For this reason, 
we study not only the general case but also the corresponding special case in detail. 
We apply our results on the linear systems to the fourth-order dispersive flows. 
We see that if the sectional curvature of the target Riemann surface is constant, 
then the equation of the dispersive flow satisfies our conditions of the well-posedness.  
\end{abstract}
\maketitle
\section{Introduction}
\label{section:introduction} 
We study the initial value problem for 
a system of fourth-order dispersive partial differential equations 
on the one-dimensional torus of the form 
\begin{alignat}{2}
  L\vec{u}
& =
  \vec{f}(t,x) 
& \quad
& \text{in}
  \quad
  \mathbb{R}\times\mathbb{T},
\label{equation:pde}      
\\
  \vec{u}(0,x)
& =
  \vec{\phi}(x)
& \quad
& \text{in}
  \quad
  \mathbb{T},
\label{equation:data}
\end{alignat}
where 
$$
L=I\frac{\partial}{\partial t}+iP,
\quad
P= ED_x^4+A(x)D_x^3+B(x)D_x^2+C(x)D_x+D(x), 
$$
$\vec{u}(t,x)={}^t[u_1(t,x),u_2(t,x)]$ 
is a $\mathbb{C}^2$-valued unknown function of 
$(t,x)\in\mathbb{R}\times\mathbb{T}$, 
$\mathbb{T}=\mathbb{R}/2\pi\mathbb{Z}$, 
$i$ is the unit of imaginary numbers, 
$D_x=-i\partial/\partial x$, 
$\vec{\phi}(x)={}^t[\phi_1(x),\phi_2(x)]$ 
and 
$\vec{f}(t,x)={}^t[f_1(t,x),f_2(t,x)]$ 
are given functions, 
$$
I
=
\begin{bmatrix}
1 & 0 \\ 0 & 1
\end{bmatrix},
\quad
E
=
\begin{bmatrix}
1 & 0 \\ 0 & -1
\end{bmatrix},
$$
$$
A(x)
=
\begin{bmatrix}
a_{11}(x) & a_{12}(x) \\ a_{21}(x) & a_{22}(x)
\end{bmatrix},
\quad
B(x)
=
\begin{bmatrix}
b_{11}(x) & b_{12}(x) \\ b_{21}(x) & b_{22}(x)
\end{bmatrix},
$$
$$
C(x)
=
\begin{bmatrix}
c_{11}(x) & c_{12}(x) \\ c_{21}(x) & c_{22}(x)
\end{bmatrix},
\quad
D(x)
=
\begin{bmatrix}
d_{11}(x) & d_{12}(x) \\ d_{21}(x) & d_{22}(x)
\end{bmatrix},
$$
$a_{jk}(x), b_{jk}(x), c_{jk}(x) \in C^\infty(\mathbb{T})$, 
and 
$C^\infty(\mathbb{T})$ is the set of all complex-valued smooth functions on $\mathbb{T}$. 
\par
The present paper is mainly concerned with 
the well-posedness of the initial value problem 
\eqref{equation:pde}-\eqref{equation:data}. 
A non-Kowalewskian is said to be dispersive-type if the initial value problem for it 
is expected to be well-posed in both directions in time. 
There are many papers studying the well-posedness of the initial value problem 
for dispersive equations. 
Unfortunately, however, the results on the necessary and sufficient conditions 
of the well-posedness are limited to 
the Schr\"odinger evolution equations on the torus, 
and one-dimensional cases. 
See 
\cite{chihara1}, 
\cite{chihara3}, 
\cite{mizohata}, 
\cite{mizuhara}, 
\cite{takeuchi}, 
\cite{tarama0} 
and 
\cite{tarama}. 
Generally speaking, 
if there exists a trapped classical orbit generated by the principal symbol, 
then the local smoothing effect of solutions of dispersive equations breaks down. 
See \cite{doi} for instance. 
In particular, if the domain of the space variables is compact, 
the smoothing effect does not occurs at all. 
For this reason, 
in case of the torus, 
restrictions on the equations for the well-posedness becomes stronger, 
and it is relatively easy to obtain the necessary and sufficient conditions. 
\par
Here we change the subject. 
In the last decade the geometric analytic studies on dispersive flows between manifolds 
have been relatively attractive in mathematics. 
Most of the equations are originated in classical mechanics. 
Some results were obtained by the geometric point of view, 
and another ones were based on the analytic approach. 
In any case, most of the results are concerned with 
the relationship between the geometric settings and the structure of the equations.  
See, e.g., 
\cite{CSU}, 
\cite{chihara2}, 
\cite{CO1}, 
\cite{CO2}, 
\cite{koiso1}, 
\cite{koiso2}, 
\cite{koiso3}, 
\cite{onodera0},  
\cite{onodera1},  
\cite{onodera11}, 
\cite{onodera12}  
and references therein. 
\par
The system \eqref{equation:pde} is related with 
a fourth-order dispersive flow for closed curves into compact Riemann surfaces of the form 
\begin{equation}
u_t
=
a\tilde{J}(u)\nabla_x^3u_x
+
\{1+bg_u(u_x,u_x)\}\tilde{J}(u)\nabla_xu_x
+
cg_u(\nabla_xu_x,u_x)\tilde{J}(u)u_x
\quad
\text{in}
\quad
\mathbb{R}\times\mathbb{T},
\label{equation:onoderasystem} 
\end{equation}
where 
$\mathbb{R}\times\mathbb{T}\ni(t,x) \mapsto u(t,x){\in}N$, 
$(N,\tilde{J},g)$ is a compact Riemann surface 
with a complex structure $\tilde{J}$ and a K\"ahler metric $g$,  
$u_t=du(\partial/\partial t)$, 
$u_x=du(\partial/\partial x)$, 
$du$ is the differential of the mapping $u$, 
$\nabla$ is the induced connection for the Levi-Civita connection $\nabla^N$ of $(N,\tilde{J},g)$, 
$a\in\mathbb{R}\setminus\{0\}$ and $b,c\in\mathbb{R}$ are constants. 
Note that 
$C(t)=\{u(t,x)\ \vert \ x\in\mathscr{T}\}$  
is a closed curve on $N$ for any fixed $t\in\mathbb{R}$, 
and 
$u$ describes the motion of a closed curve subject to the equation \eqref{equation:onoderasystem}. 
From a point of view of linear partial differential equations, 
the equation \eqref{equation:onoderasystem} has a loss of derivative of order one,  
and the classical energy estimates of solutions never work. 
Moreover, no smoothing effect of solutions can be expected 
since the sauce of the mapping is a compact space $\mathbb{T}$. 
Recently, inspite of this difficulty, 
Onodera (\cite{onodera2}) has been studying 
the initial value problem for \eqref{equation:onoderasystem}, 
and succeeded in the construction of time-local solutions. 
Unfortunately, however, his approach is based on massive and complicated computations, 
and is not comprehensive. 
In other words, it is very difficult to understand 
how he can resolve the loss of derivative of order one. 
\par
The purpose of this paper is to give the necessary and sufficient conditions of 
the well-posedness of the initial value problem \eqref{equation:pde}-\eqref{equation:data}, 
and to have insight into the structure of \eqref{equation:onoderasystem}. 
Indeed, we also study the initial value problem for a special system of the form 
\begin{alignat}{2}
  \mathscr{L}\vec{w}
& =
  \vec{h}(t,x) 
& \quad
& \text{in}
  \quad
  \mathbb{R}\times\mathbb{T},
\label{equation:pde2}      
\\
  \vec{w}(0,x)
& =
  \vec{w}_0(x)
& \quad
& \text{in}
  \quad
  \mathbb{T},
\label{equation:data2}
\end{alignat}
where 
$$
\mathscr{L}
=
I
\frac{\partial}{\partial t}
+
J
\frac{\partial^4}{\partial x^4}
+
\beta(x)
\frac{\partial^2}{\partial x^2}
+
\gamma(x)
\frac{\partial}{\partial x},
$$
$\vec{w}(t,x)={}^t[w_1(t,x),w_2(t,x)]$ 
is a $\mathbb{R}^2$-valued unknown function of 
$(t,x)\in\mathbb{R}\times\mathbb{T}$, 
$\vec{w}_0(x)$ and $\vec{h}(t,x)$ are given functions, 
$$
J
=
\begin{bmatrix}
0 & -1 \\ 1 & 0
\end{bmatrix},
\quad
\beta(x)
=
\begin{bmatrix}
\beta_{11}(x) & \beta_{12}(x) \\ \beta_{21}(x) & \beta_{22}(x)
\end{bmatrix}, 
\quad
\gamma(x)
=
\begin{bmatrix}
\gamma_{11}(x) & \gamma_{12}(x) \\ \gamma_{21}(x) & \gamma_{22}(x)
\end{bmatrix},
$$
and $\beta_{jk}(x)$ and $\gamma_{jk}(x)$ 
are real-valued smooth functions on $\mathbb{T}$. 
The system \eqref{equation:pde2} is a special case of \eqref{equation:pde}, 
and very closed to \eqref{equation:onoderasystem}. 
Obviously, the prospect of the analysis of the differential operator $\mathscr{L}$ 
is bad since its principal part has only the off-diagonal components. 
Generally speaking, however, studies on such real-valued linear systems are useful 
for solving the initial value problem for dispersive flows into almost Hermitian manifolds. 
Indeed, if the matrix like $J$ is replaced by the almost complex structure, 
the methods established for the systems are applicable to some dispersive flows 
only with minor changes in many cases.   
The necessary and sufficient conditions of $L^2$-well-posedness of 
the initial value problem \eqref{equation:pde2}-\eqref{equation:data2} 
are reduced to those of \eqref{equation:pde}-\eqref{equation:data}. 
Moreover, we give the direct proof of the sufficiency of the well-posedness of 
\eqref{equation:pde2}-\eqref{equation:data2} 
with application to \eqref{equation:onoderasystem} in mind. 
We remark that the sufficiency of our conditions on the well-posedness of 
\eqref{equation:pde}-\eqref{equation:data} 
work also if the coefficients depend on $t$. 
Finally, we introduce a moving frame along the curve described by $u(t,\cdot)$ on $N$ 
for each $t$, and obtain an $\mathbb{R}^2$-valued system 
from the equation of a higher order spatial derivative of $u$. 
We see that if the sectional curvature of the target Riemann surface is constant, 
then this system satisfies the the sufficient conditions of $L^2$-well-posedness. 
We believe that our approach in the present paper 
will give a perspective to \cite{onodera2}. 
\par
To state our results on the well-posedness of 
\eqref{equation:pde}-\eqref{equation:data}, 
we introduce some function spaces. 
We denote by $L^2(\mathbb{T};\mathbb{C}^2)$ 
the set of $\mathbb{C}^2$-valued square integrable functions on $\mathbb{T}$. 
We denote by 
$C\bigl(\mathbb{R};L^2(\mathbb{T};\mathbb{C}^2)\bigr)$ 
and 
$L^1_{\text{loc}}\bigl(\mathbb{R};L^2(\mathbb{T};\mathbb{C}^2)\bigr)$ 
the set of all $L^2(\mathbb{T};\mathbb{C}^2)$-valued continuous functions on $\mathbb{T}$, 
and the set of all $L^2(\mathbb{T};\mathbb{C}^2)$-valued 
locally integrable functions on $\mathbb{T}$ respectively. 
Our main results are the following. 
\begin{theorem}
\label{theorem:main} 
The following conditions {\rm (I)} and {\rm (II)} are mutually equivalent. 
\begin{itemize}
\item[{\rm (I)}] 
The initial value problem \eqref{equation:pde}-\eqref{equation:data} is $L^2$-well-posed, 
that is, for any $\vec{\phi} \in L^2(\mathbb{T};\mathbb{C}^2)$ 
and for any $\vec{f} \in L^1_{\text{\rm loc}}\bigl(\mathbb{R}; L^2(\mathbb{T};\mathbb{C}^2)\bigr)$, 
\eqref{equation:pde}-\eqref{equation:data} has a unique solution 
$\vec{u} \in C\bigl(\mathbb{R}; L^2(\mathbb{T};\mathbb{C}^2)\bigr)$.  
\item[{\rm (II)}] 
The coefficients 
$a_{jk}(x)$, $b_{jk}(x)$ and $c_{jk}(x)$ satisfy the following conditions  
\begin{align}
  \operatorname{Im} 
  \int_0^{2\pi}
  a_{11}(x)
  dx
& =
  0,
\label{equation:condition01}
\\
  \operatorname{Im} 
  \int_0^{2\pi}
  a_{22}(x)
  dx
& =
  0,
\label{equation:condition02}
\end{align}
\begin{align}
  \operatorname{Im} 
  \int_0^{2\pi}
  \left\{
  b_{11}(x)
  -
  \frac{3a_{11}(x)^2-4a_{12}(x)a_{21}(x)}{8}
  \right\}
  dx
& =
  0,
\label{equation:condition03}
\\
  \operatorname{Im} 
  \int_0^{2\pi}
  \left\{
  b_{22}(x)
  +
  \frac{3a_{22}(x)^2-4a_{12}(x)a_{21}(x)}{8}
  \right\}
  dx
& =
  0,
\label{equation:condition04}
\end{align}
\begin{align}
  \operatorname{Im} 
  \int_0^{2\pi}
  \biggl\{
  c_{11}(x)
& +
  \frac{i}{2}
  a_{12}^\prime(x)a_{21}(x)
\nonumber
\\
& -
  \frac{a_{11}(x)b_{11}(x)-a_{12}(x)b_{21}(x)-a_{21}(x)b_{12}(x)}{2}
\nonumber
\\
& -
  \frac{a_{11}(x)^3+4a_{11}(x)a_{12}(x)a_{21}(x)-2a_{12}(x)a_{21}(x)a_{22}(x)}{8}
  \biggr\}
  dx
  =
  0, 
\label{equation:condition05} 
\\
  \operatorname{Im} 
  \int_0^{2\pi}
  \biggl\{
  c_{22}(x)
& -
  \frac{i}{2}
  a_{12}(x)a_{21}^\prime(x)
\nonumber
\\
& +
  \frac{a_{22}(x)b_{22}(x)-a_{12}(x)b_{21}(x)-a_{21}(x)b_{12}(x)}{2}
\nonumber
\\
& +
  \frac{a_{22}(x)^3-4a_{12}(x)a_{21}(x)a_{22}(x)+2a_{11}(x)a_{12}(x)a_{21}(x)}{8}
  \biggr\}
  dx
  =
  0.
\label{equation:condition06} 
\end{align}
\end{itemize}
\end{theorem}
We essentially diagonalize our system \eqref{equation:pde} 
by an appropriate system of pseudodifferential operators, 
and the proof of Theorem~\ref{theorem:main} 
is reduced to Mizuhara's results on single equations of the form 
\begin{equation}
\frac{\partial v}{\partial t}
\pm
iD_x^4v
+
ia(x)D_x^3v
+
ib(x)D_x^2v
+
ic(x)D_xv
+
id(x)v
=
g(t,x)
\quad
\text{in}
\quad 
\mathbb{R}\times\mathbb{T},
\label{equation:single}
\end{equation}
where $v(t,x)$ is a complex-valued unknown function, 
$a(x)$, 
$b(x)$, 
$c(x)$, 
$d(x)$ 
and 
$g(t,x)$ 
are given functions. 
In \cite{mizuhara} he proved the following. 
\begin{theorem}[Mizuhara, \cite{mizuhara}]
\label{theorem:mizuhara} 
The initial value problem for \eqref{equation:single} is $L^2$-well-posed if and only if 
\begin{equation}
\operatorname{Im} 
\int_0^{2\pi}
a(x)
dx
=
0,
\label{equation:mizuhara1} 
\end{equation}
\begin{equation}
\operatorname{Im} 
\int_0^{2\pi}
\left\{
b(x)
\mp
\frac{3}{8}
a(x)^2
\right\}
dx
=
0,
\label{equation:mizuhara2} 
\end{equation}
\begin{equation}
\operatorname{Im} 
\int_0^{2\pi}
\left\{
c(x)
\mp
\frac{a(x)b(x)}{2}
\mp
\frac{a(x)^3}{8}
\right\}
dx
=
0.
\label{equation:mizuhara3} 
\end{equation}
Here we used the double-sign corresponds. 
\end{theorem}
If we consider the system for ${}^t[w_1+iw_2,w_1-iw_2]$ instead of $\vec{w}$, 
Theorem~\ref{theorem:main} implies 
the necessary and sufficient conditions of $L^2$-well-posedness of 
\eqref{equation:pde2}-\eqref{equation:data2}. 
\begin{theorem}
\label{theorem:main2}
The initial value problem \eqref{equation:pde2}-\eqref{equation:data2} 
is $L^2$-well-posed if and only if 
\begin{align}
  \operatorname{Im} 
  \int_0^{2\pi}
  \operatorname{tr}\bigl(\beta(x)\bigr)
  dx
& =
  0,
\label{equation:condition11}
\\
  \operatorname{Im} 
  \int_0^{2\pi}
  \operatorname{tr}\bigl(J\gamma(x)\bigr)
  dx
& =
  0, 
\label{equation:condition12} 
\end{align}
where 
$\operatorname{tr}\bigl(\beta(x)\bigr)=\beta_{11}(x)+\beta_{22}(x)$ 
and 
$\operatorname{tr}\bigl(J\gamma(x)\bigr)=\gamma_{12}(x)-\gamma_{21}(x)$.  
\end{theorem}
We will check that Theorem~\ref{theorem:main} implies Theorem~\ref{theorem:main2}, 
and prove the sufficiency in Theorem~\ref{theorem:main2} directly 
with the applications to \eqref{equation:onoderasystem} in mind. 
Our direct proof of the sufficiency in Theorem~\ref{theorem:main2} works 
also in the case that the coefficients $\beta_{jk}$ and $\gamma_{jk}$ 
are $C^\infty(\mathbb{T})$-valued $C^1$-functions in time. 
\par
Finally we introduce a moving frame along the curve $u(t,\cdot)$ on $TN$, 
and consider $\nabla_x^lu_x$, which is $l$-th order derivative of $u$ in $x$ 
for some large integer $l$. 
The system \eqref{equation:corrected} 
for the two components of $\nabla_x^lu_x$ in the moving frame satisfies  
a fourth-order dispersive system like \eqref{equation:pde2}. 
We see that this system satisfies the conditions 
\eqref{equation:condition11} and \eqref{equation:condition12} 
in some sense provided that the sectional curvature of the Riemann surface 
$(N,\tilde{J},g)$ is constant. 
Such geometric reductions originated from the pioneering work of 
Chang, Shatah and Uhlenbeck in \cite{CSU}. 
They constructed a moving frame along solutions 
to the one-dimensional Schr\"odinger map equation 
$u_t=\tilde{J}(u)\nabla_xu_x$ into compact Riemann surfaces, 
and obtained a complex-valued equation from the equation for $u_x$. 
Being inspired with \cite{CSU}, Onodera studied the reduction of 
third and fourth-order one-dimensional dispersive flows in \cite{onodera1}. 
Unfortunately, it is not easy to understand the structure of the modified equations for $u_x$ 
from the point of view of linear partial differential equations. 
In other words, it is hard to distinguish an unknown from coefficients 
since both of them consist of the original unknown function $u$. 
In the present paper we obtain the system 
for the higher order spatial derivative of $u$ instead of $u_x$. 
We believe that our reduction is more comprehensive 
to understand the relationship between the structure of the equation and the geometric settings.   
\par
The plan of the present paper is as follows. 
In Section~\ref{section:proof1} we shall prove Theorem~\ref{theorem:main}. 
In Section~\ref{section:proof2} we shall prove Theorem~\ref{theorem:main2}.   
Finally, in  Section~\ref{section:frame} we shall introduce the moving frame,  
and study the system for higher order spatial derivatives of $u$.   
\section{Proof of Theorem~\ref{theorem:main}}
\label{section:proof1}
In this section we prove Theorem~\ref{theorem:main}. 
We modify the system \eqref{equation:pde} by using bounded pseudodifferential operators. 
Then, the system \eqref{equation:pde} becomes a pair of single equations essentially, 
and the proof of Theorem~\ref{theorem:main} is reduced to Theorem~\ref{theorem:mizuhara}. 
We make use of elementary pseudodifferential calculus on $\mathbb{R}$. 
See, e.g., \cite{kumano-go}, \cite{NR} and \cite{taylor} for this.   
\par
Let $m\in\mathbb{R}$. 
We denote by $S^m(\mathbb{T})$ the set of all smooth functions 
$q(x,\xi)$ of $(x,\xi)\in\mathbb{T}\times\mathbb{R}$ 
with the following properties: 
for all nonnegative integers $\alpha$ and $\beta$, 
there exists a positive constant $C_{\alpha\beta}$ such that 
$$
\left\lvert
\frac{\partial^{\alpha+\beta} q}{\partial x^\beta \partial \xi^\alpha}(x,\xi)
\right\rvert
\leqslant
C_{\alpha\beta}
\bigl(1+\lvert\xi\rvert\bigr)^{m-\alpha}
$$
holds for all $(x,\xi)\in\mathbb{T}\times\mathbb{R}$. 
Let $\mathscr{S}(\mathbb{R})$ be the Schwartz class of 
rapidly decreasing functions on $\mathbb{R}$. 
A linear operator $Q:\mathscr{S}(\mathbb{R}) \rightarrow \mathscr{S}(\mathbb{R})$ 
is said to be a pseudodifferential operator with a symbol $q$ if 
$$
Qu(x)
=
\frac{1}{2\pi}
\iint_{\mathbb{R}\times\mathbb{R}}
e^{i(x-y)\xi}
q(x,\xi)
u(y)
dyd\xi
$$
for $u\in\mathscr{S}(\mathbb{R})$.  
We remark that the operator $Q$ can be exteded as 
$Q: C^\infty(\mathbb{T}) \rightarrow C^\infty(\mathbb{T})$ 
since $q(x,\xi)$ is $2\pi$-periodic in $x$. 
Conversely, 
if a linear operator $Q: C^\infty(\mathbb{T}) \rightarrow C^\infty(\mathbb{T})$ is given, 
its symbol $\sigma(Q)$ is computed by 
$\sigma(Q)(x,\xi)=e^{-ix\xi}Qe^{ix\xi}$. 
In what follows we mainly deal with $2\times2$-matrix-valued symbols. 
The classes of such symbols are denoted by $S^m\bigl(\mathbb{T};M(2)\bigr)$.  
\begin{proof}[{\bf Proof of Theorem~\ref{theorem:main}}]
We split the proof into several steps in order to make our pseudodifferential calculus 
simple and comprehensive. 
We mainly use pseudodifferential operators of negative orders $-1$, $-2$ and $-3$. 
We apply these operators to \eqref{equation:pde} step by step. 
Let $r$ be a sufficiently large positive number. 
Pick up a smooth function $\varphi_r(\xi)$ on $\mathbb{R}$ such that 
$0\leqslant\varphi_r(\xi)\leqslant1$, 
$\varphi_r(\xi)=1$ for $\lvert\xi\rvert\geqslant r+1$, 
$\varphi_r(\xi)=0$ for $\lvert\xi\rvert\leqslant r$ 
and 
$\varphi_r(\xi)=\varphi(-\xi)$. 
The last property $\varphi_r(\xi)=\varphi(-\xi)$ is not necessary in this section, 
but will be essentially used in the next section. 
For $2\times2$ matrices, we use the following notation:
$$
A^{\text{diag}}(x)
=
\begin{bmatrix}
a_{11}(x) & 0
\\
0 & a_{22}(x) 
\end{bmatrix}, 
\quad
A^{\text{off}}(x)
=
\begin{bmatrix}
0 & a_{12}(x)
\\
a_{21}(x) & 0
\end{bmatrix} 
$$
for 
$$
A(x)
=
\begin{bmatrix}
a_{11}(x) & a_{12}(x)
\\
a_{21}(x) & a_{22}(x) 
\end{bmatrix}.  
$$
The set of all bounded linear operators on $L^2(\mathbb{T};\mathbb{C}^2)$ 
is denoted by $\mathscr{L}\bigl(L^2(\mathbb{T};\mathbb{C}^2)\bigr)$, 
and its norm is denoted by $\lVert\cdot\rVert$.  
\vspace{11pt}
\\
{\bf Step 1}: 
diagonalization of $iA(x)D_x^3$. 
Note that $EA^{\text{off}}(x)+A^{\text{off}}(x)E=0$. 
We define a pseudodifferential operator $\tilde{\Lambda}_1$ of order $-1$ by 
$$
\sigma\bigl(\tilde{\Lambda}_2\bigr)(x,\xi)
=
\frac{1}{2}
E
A^{\text{off}}(x)
\frac{\varphi_r(\xi)}{\xi}
=
-
\frac{1}{2}
A^{\text{off}}(x)
E
\frac{\varphi_r(\xi)}{\xi}. 
$$
Set $\Lambda_1=I+\tilde{\Lambda}_1$. 
Then 
$\sigma\bigl(\tilde{\Lambda}_1\bigr)(x,\xi) \in S^{-1}\bigl(\mathbb{T};M(2)\bigr)$ 
and  
$\sigma\bigl(\Lambda_1\bigr)(x,\xi) \in S^{0}\bigl(\mathbb{T};M(2)\bigr)$. 
Note that $\lVert\tilde{\Lambda}_1\rVert=\mathcal{O}(1/r)$. 
If we take a sufficiently large $r>0$, 
then $\Lambda_1$ is invertible on $\mathscr{L}\bigl(L^2(\mathbb{T};\mathbb{C}^2)\bigr)$, 
and the inverse is given by a Neumann series 
\begin{equation}
\Lambda_1^{-1}
=
I
+
\sum_{l=1}^\infty
\bigl\{-\tilde{\Lambda}_1\bigr\}^l
=
I
-
\tilde{\Lambda}_1
+
\tilde{\Lambda}_1^2
-
\tilde{\Lambda}_1^3
+
\tilde{\Lambda}_1^4\Lambda_1^{-1}. 
\label{equation:inverse1}
\end{equation}
Set $L_1=\Lambda_1L\Lambda_1^{-1}$ and $P_1=\Lambda_1P\Lambda_1^{-1}$ for short. 
Then we have 
$$
\Lambda_1L
=
L_1\Lambda_1
=
\left\{
I\frac{\partial}{\partial t}+iP_1
\right\}\Lambda_1. 
$$
We compute $P_1=\Lambda_1P\Lambda_1^{-1}$. 
Using \eqref{equation:inverse1}, we have 
\begin{align}
  P_1
& =
  \bigl(I+\tilde{\Lambda}_1\bigr)
  P
  \bigl(
  I
  -
  \tilde{\Lambda}_1
  +
  \tilde{\Lambda}_1^2
  -
  \tilde{\Lambda}_1^3
  +
  \tilde{\Lambda}_1^4\Lambda_1^{-1}
  \bigr)
\nonumber
\\
& \equiv
  P
  +
  \bigl(-P\tilde{\Lambda}_1+\tilde{\Lambda}_1P\bigr)
  -
  \bigl(-P\tilde{\Lambda}_1+\tilde{\Lambda}_1P\bigr)\tilde{\Lambda}_1
  +
  \bigl(-P\tilde{\Lambda}_1+\tilde{\Lambda}_1P\bigr)\tilde{\Lambda}_1^2
\label{equation:p1}
\end{align}
modulo $\mathscr{L}\bigl(L^2(\mathbb{T};\mathbb{C}^2)\bigr)$.
We compute the common term $-P\tilde{\Lambda}_1+\tilde{\Lambda}_1P$. 
Since $\tilde{\Lambda}_1$ is a pseudodifferential operator of order $-1$, 
we deduce that 
\begin{align}
  -P\tilde{\Lambda}_1+\tilde{\Lambda}_1P
  \equiv
& -ED_x^4\tilde{\Lambda}_1+\tilde{\Lambda}_1ED_x^4
\label{equation:part201}
\\
& -A(x)D_x^3\tilde{\Lambda}_1+\tilde{\Lambda}_1A(x)D_x^3
\label{equation:part202}
\\
& -B(x)D_x^2\tilde{\Lambda}_1+\tilde{\Lambda}_1B(x)D_x^2
  \qquad\text{mod}\quad
  \mathscr{L}\bigl(L^2(\mathbb{T};\mathbb{C}^2)\bigr).  
\label{equation:part203}   
\end{align} 
We compute the above term by term. 
Here we recall the definition of $\tilde{\Lambda}_1$, and the identities  
$E^2=I$ and $A^{\text{off}}(x)E+EA^{\text{off}}(x)=0$. 
We deduce that  
\begin{align*}
  \sigma\bigl(\text{\eqref{equation:part201}}\bigr)(x,\xi)
& \equiv
  \sum_{k=0}^2
  \frac{(-i)^k}{k!}
  \biggl[
  \left\{
  -
  \frac{\partial^k}{\partial \xi^k}
  E
  \xi^4
  \right\}
  \left\{
  \frac{\partial^k}{\partial x^k}
  \frac{1}{2}
  E
  A^{\text{off}}(x)
  \frac{\varphi_r(\xi)}{\xi}
  \right\}
\\
& \qquad\qquad\quad
  +
  \left\{
  -
  \frac{\partial^k}{\partial \xi^k}
  \frac{1}{2}
  A^{\text{off}}(x)
  E
  \frac{\varphi_r(\xi)}{\xi}
  \right\}
  \left\{
  \frac{\partial^k}{\partial x^k}
  E
  \xi^4
  \right\}
  \biggr]
\\
& =
  -\frac{1}{2}
  \sum_{k=0}^2
  \frac{(-i)^k}{k!}
  \biggl\{
  \frac{\partial^k A^{\text{off}}}{\partial x^k}(x)
  \frac{\varphi_r(\xi)}{\xi}
  \left(\frac{\partial^k}{\partial \xi^k}\xi^4\right)
\\
& \qquad\qquad\qquad\qquad
  +
  A^{\text{off}}(x)
  \left(\frac{\partial^k}{\partial \xi^k}\frac{\varphi_r(\xi)}{\xi}\right)
  \left(\frac{\partial^k}{\partial x^k}\xi^4\right)
  \biggr\}
\\
& \equiv
  -A^{\text{off}}(x)\xi^3
  +
  2i\frac{\partial A^{\text{off}}}{\partial x}(x)\xi^2
  +
  3\frac{\partial^2 A^{\text{off}}}{\partial x^2}(x)\xi,
\\
  \sigma\bigl(\text{\eqref{equation:part202}}\bigr)(x,\xi)
& \equiv
  \sum_{k=0}^1
  \frac{(-i)^k}{k!}
  \biggl[
  \left\{
  -\frac{\partial^k}{\partial \xi^k}
  A(x)\xi^3
  \right\}
  \left\{
  \frac{\partial^k}{\partial x^k}
  \frac{1}{2}
  E
  A^{\text{off}}(x)
  \frac{\varphi_r(\xi)}{\xi}
  \right\}
\\
& \qquad\qquad\quad
  +
  \left\{
  -
  \frac{\partial^k}{\partial \xi^k}
  \frac{1}{2}
  A^{\text{off}}(x)
  E
  \frac{\varphi_r(\xi)}{\xi}
  \right\}
  \left\{
  \frac{\partial^k}{\partial x^k}
  A(x)\xi^3
  \right\}
  \biggr]
\\
& \equiv
  \left\{
  -
  \frac{1}{2}
  A(x)EA^{\text{off}}(x)
  -
  \frac{1}{2}
  A^{\text{off}}(x)EA(x)  
  \right\}
  \xi^2
\\
& +
  \left\{
  \frac{3i}{2}
  A(x)E\frac{\partial A^{\text{off}}}{\partial x}(x)
  -
  \frac{i}{2}
  A^{\text{off}}(x)E\frac{\partial A}{\partial x}(x)
  \right\}
  \xi,
\\
  \sigma\bigl(\text{\eqref{equation:part203}}\bigr)(x,\xi)
& \equiv
  \left\{
  -
  \frac{1}{2}
  B(x)EA^{\text{off}}(x)
  -
  \frac{1}{2}
  A^{\text{off}}(x)EB(x)
  \right\}
  \xi,
\end{align*}
modulo $S^0\bigl(\mathbb{T};M(2)\bigr)$. 
Combining the above, we obtain 
\begin{align}
& \sigma\bigl(-P\tilde{\Lambda}_1+\tilde{\Lambda}_1P\bigr)(x,\xi)
\nonumber
\\
& \equiv 
  -A^{\text{off}}(x)\xi^3
\nonumber
\\
& +
  \left\{
  2i\frac{\partial A^{\text{off}}}{\partial x}(x)
  -
  \frac{1}{2}
  A(x)EA^{\text{off}}(x)
  -
  \frac{1}{2}
  A^{\text{off}}(x)EA(x)  
  \right\}
  \xi^2
\nonumber
\\ 
& +
  \biggl\{
  3\frac{\partial^2 A^{\text{off}}}{\partial x^2}(x)
  +
  \frac{3i}{2}
  A(x)E\frac{\partial A^{\text{off}}}{\partial x}(x)
  -
  \frac{i}{2}
  A^{\text{off}}(x)E\frac{\partial A}{\partial x}(x)
\nonumber
\\
& \qquad\qquad
  -
  \frac{1}{2}
  B(x)EA^{\text{off}}(x)
  -
  \frac{1}{2}
  A^{\text{off}}(x)EB(x)  
  \biggr\}
  \xi,
\label{equation:part204}
\end{align}
modulo $S^0\bigl(\mathbb{T};M(2)\bigr)$. 
By using \eqref{equation:part204}, we deduce that 
\begin{align}
& \sigma\Bigl(-\bigl(-P\tilde{\Lambda}_1+\tilde{\Lambda}_1P\bigr)\tilde{\Lambda}_1\Bigr)(x,\xi)
\nonumber
\\
& \equiv
  \sum_{k=0}^1
  \frac{(-i)^k}{k!}
  \left\{
  -
  \frac{\partial^k}{\partial \xi^k}
  \sigma\bigl(-P\tilde{\Lambda}_1+\tilde{\Lambda}_1P\bigr)(x,\xi)
  \right\}
  \left\{
  \frac{\partial^k}{\partial x^k}
  \frac{1}{2}EA^{\text{off}}(x)
  \frac{\varphi_r(\xi)}{\xi}
  \right\}
\nonumber 
\\
& \equiv
  \frac{1}{2}
  A^{\text{off}}(x)EA^{\text{off}}(x)
  \xi^2
\nonumber
\\
& +
  \biggl\{
  -
  \frac{3i}{2}
  A^{\text{off}}(x)E\frac{\partial A^{\text{off}}}{\partial x}(x)
  -
  i
  \frac{\partial A^{\text{off}}}{\partial x}(x)EA^{\text{off}}(x)
\nonumber
\\
& \qquad
  \frac{1}{4}
  A(x)EA^{\text{off}}(x)EA^{\text{off}}(x)
  +
  \frac{1}{4}
  A^{\text{off}}(x)EA(x)EA^{\text{off}}(x)
  \biggr\}
  \xi,
\label{equation:part205}
\\
& \sigma\Bigl(\bigl(-P\tilde{\Lambda}_1+\tilde{\Lambda}_1P\bigr)\tilde{\Lambda}_1^2\Bigr)(x,\xi)
\nonumber
\\
& \equiv
  \sigma\bigl(-P\tilde{\Lambda}_1+\tilde{\Lambda}_1P\bigr)(x,\xi)
  \left\{
  \frac{1}{2}EA^{\text{off}}(x)
  \frac{\varphi_r(\xi)}{\xi}
  \right\}^2
\nonumber
\\
& \equiv
  -
  \frac{1}{4}
  A^{\text{off}}(x)EA^{\text{off}}(x)EA^{\text{off}}(x)
  \xi, 
\label{equation:part206}
\end{align}
modulo $S^0\bigl(\mathbb{T};M(2)\bigr)$. 
Substituting 
\eqref{equation:part204}, 
\eqref{equation:part205} 
and  
\eqref{equation:part206} 
into 
\eqref{equation:p1}, 
we obtain 
\begin{equation}
P_1
\equiv
ED_x^4
+
A^{\text{diag}}(x)D_x^3
+
B_1(x)D_x^2
+
C_1(x)D_x
\qquad\text{mod}\quad 
\mathscr{L}\bigl(L^2(\mathbb{T};\mathbb{C}^2)\bigr), 
\label{equation:P1} 
\end{equation}
\begin{align}
  B_1(x)
& =
  B(x)
  +
  2i
  \frac{\partial A^{\text{off}}}{\partial x}(x)
\nonumber
\\
& +
  \frac{1}{2}
  A^{\text{off}}(x)EA^{\text{off}}(x)
  -
  \frac{1}{2}
  A(x)EA^{\text{off}}(x)
  -
  \frac{1}{2}
  A^{\text{off}}(x)EA(x),
\label{equation:B1}
\\
  C_1(x)
& =
  C(x)
  +
  3
  \frac{\partial^2 A^{\text{off}}}{\partial x^2}(x)
\nonumber
\\
& +
  \frac{3i}{2}
  A(x)E\frac{\partial A^{\text{off}}}{\partial x}(x)
  -
  \frac{i}{2}
  A^{\text{off}}(x)E\frac{\partial A}{\partial x}(x)
\nonumber
\\
& -
  \frac{3i}{2}
  A^{\text{off}}(x)E\frac{\partial A^{\text{off}}}{\partial x}(x)
  -
  i
  \frac{\partial A^{\text{off}}}{\partial x}(x)EA^{\text{off}}(x)
\nonumber
\\
& +
  \frac{1}{4}
  A(x)EA^{\text{off}}(x)EA^{\text{off}}(x)
  +
  \frac{1}{4}
  A^{\text{off}}(x)EA(x)EA^{\text{off}}(x)
\nonumber
\\
& -
  \frac{1}{4}
  A^{\text{off}}(x)EA^{\text{off}}(x)EA^{\text{off}}(x)
\nonumber
\\
& -
  \frac{1}{2}
  B(x)EA^{\text{off}}(x)
  -
  \frac{1}{2}
  A^{\text{off}}(x)EB(x).
\label{equation:C1}
\end{align}
\vspace{11pt}
\\
{\bf Step 2}: 
diagonalization of $iB_1(x)D_x^2$. 
Note that $EB_1^{\text{off}}(x)+B_1^{\text{off}}(x)E=0$. 
We define a pseudodifferential operator $\tilde{\Lambda}_2$ of order $-2$ by 
$$
\sigma\bigl(\tilde{\Lambda}_2\bigr)(x,\xi)
=
\frac{1}{2}
E
B_1^{\text{off}}(x)
\frac{\varphi_r(\xi)}{\xi^2}
=
-
\frac{1}{2}
B_1^{\text{off}}(x)
E
\frac{\varphi_r(\xi)}{\xi^2}. 
$$
Set $\Lambda_2=I+\tilde{\Lambda}_2$. 
Then 
$\sigma\bigl(\tilde{\Lambda}_2\bigr)(x,\xi) \in S^{-2}\bigl(\mathbb{T};M(2)\bigr)$ 
and 
$\sigma\bigl(\Lambda_2\bigr)(x,\xi) \in S^{0}\bigl(\mathbb{T};M(2)\bigr)$. 
Note that $\lVert\tilde{\Lambda}_2\rVert=\mathcal{O}(1/r^2)$. 
In the same way as Step 1, 
$\Lambda_2$ is invertible on $\mathscr{L}\bigl(L^2(\mathbb{T};\mathbb{C}^2)\bigr)$, 
and the inverse is given by 
$\Lambda_2^{-1}=I-\tilde{\Lambda}_2+\tilde{\Lambda}_2^2\Lambda_2^{-1}$. 
Set $L_2=\Lambda_2L_1\Lambda_2^{-1}$ and $P_2=\Lambda_2P_1\Lambda_2^{-1}$ for short. 
Then we have 
$$
\Lambda_2L_1
=
L_2\Lambda_2
=
\left\{
I\frac{\partial}{\partial t}+iP_2
\right\}\Lambda_2. 
$$
We compute $P_2=\Lambda_2P_1\Lambda_2^{-1}=\Lambda_2\Lambda_1P\Lambda_1^{-1}\Lambda_2^{-1}$. 
By using $\Lambda_2^{-1}=I-\tilde{\Lambda}_2+\tilde{\Lambda}_2^2\Lambda_2^{-1}$, 
we have 
\begin{align}
  P_2
& =
  \bigl(I+\tilde{\Lambda}_2\bigr)
  P_1
  \bigl(
  I
  -
  \tilde{\Lambda}_2
  +
  \tilde{\Lambda}_2^2\Lambda_2^{-1}
  \bigr)
\nonumber
\\
& \equiv
  P_1
  -
  P_1\tilde{\Lambda}_2
  +
  \tilde{\Lambda}_2P_1 
\nonumber
\\
& \equiv
  P_1
  +
  \bigl(
  -
  ED_x^4\tilde{\Lambda}_2
  +
  \tilde{\Lambda}_2ED_x^4
  \bigr)
  +
  \bigl(
  -
  A^{\text{diag}}(x)D_x^3\tilde{\Lambda}_2
  +
  \tilde{\Lambda}_2A^{\text{diag}}(x)D_x^3
  \bigr) 
\label{equation:p2}
\end{align}
modulo $\mathscr{L}\bigl(L^2(\mathbb{T};\mathbb{C}^2)\bigr)$. 
Since $E^2=I$, $EB_1^{\text{off}}(x)+B_1^{\text{off}}(x)E=0$ and 
$\sigma\bigl(\tilde{\Lambda}_2\bigr)(x,\xi) \in S^{-2}\bigl(\mathbb{T};M(2)\bigr)$, 
the symbols of the second and third terms of the right hand side of \eqref{equation:p2} are  
\begin{align*}
& \sigma
  \bigl(
  -
  ED_x^4\tilde{\Lambda}_2
  +
  \tilde{\Lambda}_2ED_x^4
  \bigr)
  (x,\xi)
\\
& \equiv
  \sum_{k=0}^1
  \frac{(-i)^k}{k!}
  \biggl[
  \left\{
  -
  \frac{\partial^k}{\partial \xi^k}
  E
  \xi^4
  \right\}
  \left\{
  \frac{\partial^k}{\partial x^k}
  \frac{1}{2}
  E
  B_1^{\text{off}}(x)
  \frac{\varphi_r(\xi)}{\xi}
  \right\}
\\
& \qquad\qquad\quad
  +
  \left\{
  -
  \frac{\partial^k}{\partial \xi^k}
  \frac{1}{2}
  B_1^{\text{off}}(x)
  E
  \frac{\varphi_r(\xi)}{\xi}
  \right\}
  \left\{
  \frac{\partial^k}{\partial x^k}
  E
  \xi^4
  \right\}
  \biggr]
\\
& \equiv
  -
  B_1^{\text{off}}(x)
  \xi^2
  +
  2i
  \frac{\partial B_1^{\text{off}}}{\partial x}(x)
  \xi, 
\\
& \sigma
  \bigl(
  -
  A^{\text{diag}}(x)D_x^3\tilde{\Lambda}_2
  +
  \tilde{\Lambda}_2A^{\text{diag}}(x)D_x^3
  \bigr)
  (x,\xi)
\\
& \equiv
  \left\{
  -
  \frac{1}{2}
  A^{\text{diag}}(x)EB^{\text{off}}(x)
  -
  \frac{1}{2}
  B^{\text{off}}(x)EA^{\text{diag}}(x)
  \right\}
  \xi,
\end{align*}
modulo $S^0\bigl(\mathbb{T};M(2)\bigr)$. 
Substituting these into \eqref{equation:p2}, we obtain 
\begin{equation}
P_2
\equiv
E
D_x^4
+
A^{\text{diag}}(x)
D_x^3
+
B_1^{\text{diag}}(x)
D_x^2
+
C_2(x)
D_x
\qquad\text{mod}\quad
\mathscr{L}\bigl(L^2(\mathbb{T};\mathbb{C}^2)\bigr),
\label{equation:P2}
\end{equation}
\begin{equation}
C_2(x)
=
C_1(x)
+
2i
\frac{\partial B_1^{\text{off}}}{\partial x}(x)
-
\frac{1}{2}
A^{\text{diag}}(x)EB^{\text{off}}(x)
-
\frac{1}{2}
B^{\text{off}}(x)EA^{\text{diag}}(x).
\label{equation:C2} 
\end{equation}
\vspace{11pt}
\\
{\bf Step 3}: 
diagonalization of $iC_2(x)D_x$. 
Note that $EC_2^{\text{off}}(x)+C_2^{\text{off}}(x)E=0$. 
We define a pseudodifferential operator $\tilde{\Lambda}_3$ of order $-3$ by 
$$
\sigma\bigl(\tilde{\Lambda}_3\bigr)(x,\xi)
=
\frac{1}{2}
E
C_2^{\text{off}}(x)
\frac{\varphi_r(\xi)}{\xi^3}
=
-
\frac{1}{2}
C_2^{\text{off}}(x)
E
\frac{\varphi_r(\xi)}{\xi^3}. 
$$
Set $\Lambda_3=I+\tilde{\Lambda}_3$. 
Then 
$\sigma\bigl(\tilde{\Lambda}_3\bigr)(x,\xi) \in S^{-3}\bigl(\mathbb{T};M(2)\bigr)$ 
and 
$\sigma\bigl(\Lambda_3\bigr)(x,\xi) \in S^{0}\bigl(\mathbb{T};M(2)\bigr)$. 
Note that $\lVert\tilde{\Lambda}_3\rVert=\mathcal{O}(1/r^3)$. 
In the same way as Step 1, 
$\Lambda_3$ is invertible on $\mathscr{L}\bigl(L^2(\mathbb{T};\mathbb{C}^2)\bigr)$, 
and the inverse is given by 
$\Lambda_3^{-1}=I-\tilde{\Lambda}_3+\tilde{\Lambda}_3^2\Lambda_3^{-1}$. 
Set $L_3=\Lambda_3L_2\Lambda_3^{-1}$ and $P_3=\Lambda_3P_2\Lambda_3^{-1}$ for short. 
Then we have 
$$
\Lambda_3L_2
=
L_3\Lambda_3
=
\left\{
I\frac{\partial}{\partial t}+iP_2
\right\}\Lambda_3. 
$$
We compute 
$P_3=\Lambda_3P_2\Lambda_3^{-1}=\Lambda_3\Lambda_2\Lambda_1P\Lambda_1^{-1}\Lambda_2^{-1}\Lambda_3^{-1}$. 
By using $\Lambda_3^{-1}=I-\tilde{\Lambda}_3+\tilde{\Lambda}_3^2\Lambda_3^{-1}$, 
we have 
\begin{align}
  P_3
& =
  \bigl(I+\tilde{\Lambda}_3\bigr)
  P_2
  \bigl(
  I
  -
  \tilde{\Lambda}_3
  +
  \tilde{\Lambda}_3^2\Lambda_3^{-1}
  \bigr)
\nonumber
\\
& \equiv
  P_2
  -
  P_2\tilde{\Lambda}_3
  +
  \tilde{\Lambda}_3P_2 
\nonumber
\\
& \equiv
  P_2
  +
  \bigl(
  -
  ED_x^4\tilde{\Lambda}_3
  +
  \tilde{\Lambda}_3ED_x^4
  \bigr)
  \qquad\text{mod}\quad
  \mathscr{L}\bigl(L^2(\mathbb{T};\mathbb{C}^2)\bigr).
\label{equation:p3}
\end{align}
Since $E^2=I$, $EC_2^{\text{off}}(x)+C_2^{\text{off}}(x)E=0$ and 
$\sigma\bigl(\tilde{\Lambda}_3\bigr)(x,\xi) \in S^{-3}\bigl(\mathbb{T};M(2)\bigr)$, 
the symbol of the second term of the right hand side of \eqref{equation:p3} is   
\begin{align*}
  \sigma
  \bigl(
  -
  ED_x^4\tilde{\Lambda}_3
  +
  \tilde{\Lambda}_3ED_x^4
  \bigr)
  (x,\xi)
& \equiv
  -E\xi^4
  \cdot
  \frac{1}{2}
  E
  C_2^{\text{off}}(x)
  \frac{\varphi_r(\xi)}{\xi^3}
  -
  \frac{1}{2}
  C_2^{\text{off}}(x)
  \frac{\varphi_r(\xi)}{\xi^3}
  E
  \cdot
  E\xi^4
\\
& \equiv
  -C_2^{\text{off}}(x)\xi
  \qquad\text{mod}\quad
  S^0\bigl(\mathbb{T};M(2)\bigr). 
\end{align*}
Substituting this into \eqref{equation:p3}, we obtain 
\begin{equation}
P_3
\equiv
E
D_x^4
+
A^{\text{diag}}(x)
D_x^3
+
B_1^{\text{diag}}(x)
D_x^2
+
C_2^{\text{diag}}(x)
D_x
\qquad\text{mod}\quad
\mathscr{L}\bigl(L^2(\mathbb{T};\mathbb{C}^2)\bigr),
\label{equation:P3}
\end{equation}
This completes our diagonalization of $L$. 
\vspace{11pt}
\\
{\bf Step 4}. 
We shall complete the proof of Theorem~\ref{theorem:main}. 
Set $\Lambda=\Lambda_3\Lambda_2\Lambda_1$ for short. 
Since $\Lambda$ is an invertible bounded linear operators on 
$L^2(\mathbb{T};\mathbb{C}^2)$ 
and 
$\Lambda L \equiv L_3 \Lambda$ mod $\mathscr{L}\bigl(L^2(\mathbb{T};\mathbb{C}^2)\bigr)$, 
the $L^2$-well-posedness of 
\eqref{equation:pde}-\eqref{equation:data} 
is equivalent to that of the initial value problem of the form 
\begin{equation}
L_3\vec{u}=\vec{f}(t,x),
\quad
\vec{u}(0,x)=\vec{\phi}(x). 
\label{equation:L3}
\end{equation}
The system $L_3\vec{u}=\vec{f}(t,x)$ is exactly a pair of two single equations, 
and not a system essentially. 
Thus the $L^2$-well-posedness of \eqref{equation:L3} is reduced to 
Theorem~\ref{theorem:mizuhara}. 
It suffices to check the conditions 
\eqref{equation:mizuhara1}, 
\eqref{equation:mizuhara2} 
and 
\eqref{equation:mizuhara3}. 
By using 
\eqref{equation:P3}, 
\eqref{equation:B1}, 
\eqref{equation:C1} 
and 
\eqref{equation:C2}, 
we can obtain the concrete form of $L_3$. 
The first row of $L_3$ is 
$$
\frac{\partial}{\partial t}
+
iD_x^4
+
a_{11}(x)D_x^3
+
b_{1,11}(x)D_x^2
+
c_{2,11}(x)D_x,
$$
\begin{align*}
  b_{1,11}(x)
& =
  b_{11}(x)
  +
  \frac{1}{2}
  a_{12}(x)a_{21}(x),
\\
  c_{2,11}(x)
& =
  c_{11}(x)
  +
  \frac{i}{2}
  \bigl\{a_{12}(x)a_{21}(x)\bigr\}_x
  +
  \frac{i}{2}
  a_{12}^\prime(x)a_{21}(x)
\\
& -
  \frac{1}{4}
  a_{12}(x)a_{21}(x)
  \bigl\{a_{11}(x)-a_{22}(x)\bigr\}
\\
& +
  \frac{1}{2}
  \bigl\{a_{12}(x)b_{21}(x)+a_{21}(x)b_{12}(x)\bigr\}, 
\end{align*}
and the second row of $L_3$ is 
$$
\frac{\partial}{\partial t}
-
iD_x^4
+
a_{22}(x)D_x^3
+
b_{1,22}(x)D_x^2
+
c_{2,22}(x)D_x,
$$
\begin{align*}
  b_{1,22}(x)
& =
  b_{22}(x)
  -
  \frac{1}{2}
  a_{12}(x)a_{21}(x),
\\
  c_{2,22}(x)
& =
  c_{22}(x)
  -
  \frac{i}{2}
  \bigl\{a_{12}(x)a_{21}(x)\bigr\}_x
  +
  \frac{i}{2}
  a_{12}(x)a_{21}^\prime(x)
\\
& +
  \frac{1}{4}
  a_{12}(x)a_{21}(x)
  \bigl\{a_{11}(x)-a_{22}(x)\bigr\}
\\
& -
  \frac{1}{2}
  \bigl\{a_{12}(x)b_{21}(x)+a_{21}(x)b_{12}(x)\bigr\}.  
\end{align*}
Checking the conditions 
\eqref{equation:mizuhara1}, 
\eqref{equation:mizuhara2} 
and 
\eqref{equation:mizuhara3} 
for the above, 
we can obtain 
\eqref{equation:condition01}, 
\eqref{equation:condition02}, 
\eqref{equation:condition03}, 
\eqref{equation:condition04}, 
\eqref{equation:condition05} 
and  
\eqref{equation:condition06}. 
We omit the detail. 
\end{proof}
%
%
\section{Proof of Theorem~\ref{theorem:main2}} 
\label{section:proof2}
In this section we prove Theorem~\ref{theorem:main2}. 
Firstly we prove it by using Theorem~\ref{theorem:main}. 
Secondly we give the direct proof of its sufficiency of $L^2$-well-posedness. 
We believe that this will be helpful for studying the initial value problem for 
\eqref{equation:onoderasystem}. 
\begin{proof}[{\bf Proof of Theorem~\ref{theorem:main2}}]
Let $\vec{w}$ be a solution to \eqref{equation:pde2}-\eqref{equation:data2}. 
Let $M$ be a $2\times2$ matrix defined by 
$$
M
=
\begin{bmatrix}
1 & i
\\
1 & -i 
\end{bmatrix}. 
$$
Then 
$$
M^{-1}
=
\frac{1}{2}
\begin{bmatrix}
1 & 1
\\
-i & i 
\end{bmatrix}. 
$$
Set 
$\vec{U}=M\vec{w}$, 
$\vec{F}=M\vec{h}$, 
$\vec{U}_0=M\vec{w}_0$ 
and 
$\mathscr{L}_1=M\mathscr{L}M^{-1}$. 
Then we have 
\begin{equation}
\mathscr{L}_1\vec{U}=\vec{F},
\quad
\vec{U}(0,x)=\vec{U}_0(x).
\label{equation:301} 
\end{equation}
The $L^2$-well-posedness of \eqref{equation:pde2}-\eqref{equation:data2} 
is equivalent to that of the initial value problem for \eqref{equation:301}. 
We shall obtain the concrete form of $\mathscr{L}_1$. 
Simple computations give 
\begin{align*}
  \mathscr{L}_1
& =
  M
  \left\{
  I
  \frac{\partial}{\partial t}
  +
  J
  \frac{\partial^4}{\partial x^4}
  +
  \beta(x)
  \frac{\partial^2}{\partial x^2}
  +
  \gamma(x)
  \frac{\partial}{\partial x}
  \right\}
  M^{-1}
\\
& =
  M
  \left\{
  I
  \frac{\partial}{\partial t}
  +
  J
  D_x^4
  -
  \beta(x)
  D_x^2
  +
  i
  \gamma(x)
  D_x
  \right\}
  M^{-1}
\\
& =
  I
  \frac{\partial}{\partial t}
  +
  MJM^{-1}
  D_x^4
  -
  M\beta(x)M^{-1}
  D_x^2
  +
  i
  M\gamma(x)M^{-1}
  D_x
\\
& =
  I
  \frac{\partial}{\partial t}
  +
  i
  \bigl\{-iMJM^{-1}\bigr\}
  D_x^4
  +
  i
  \bigl\{iM\beta(x)M^{-1}\bigr\}
  D_x^2
  +
  i
  M\gamma(x)M^{-1}
  D_x
\\
& =
  I
  \frac{\partial}{\partial t}
  +
  i
  E
  D_x^4
  +
  i
  \tilde{B}(x)
  D_x^2
  +
  i
  \tilde{C}(x)
  D_x, 
\end{align*}
where
$$
\tilde{B}(x)
=
\frac{1}{2}
\begin{bmatrix}
\tilde{b}_{11}(x) & \tilde{b}_{12}(x)
\\
\tilde{b}_{21}(x) & \tilde{b}_{22}(x) 
\end{bmatrix},
\quad
\tilde{C}(x)
=
\frac{1}{2}
\begin{bmatrix}
\tilde{c}_{11}(x) & \tilde{c}_{12}(x)
\\
\tilde{c}_{21}(x) & \tilde{c}_{22}(x) 
\end{bmatrix},
$$
\begin{align*}
  \tilde{b}_{11}(x)
& =
  \bigl\{\beta_{12}(x)-\beta_{21}(x)\bigr\}+i\bigl\{\beta_{11}(x)+\beta_{22}(x)\bigr\},
\\
  \tilde{b}_{12}(x)
& =
   -\bigl\{\beta_{12}(x)+\beta_{21}(x)\bigr\}+i\bigl\{\beta_{11}(x)-\beta_{22}(x)\bigr\},
\\
  \tilde{b}_{21}(x)
& =
  \bigl\{\beta_{12}(x)+\beta_{21}(x)\bigr\}+i\bigl\{\beta_{11}(x)-\beta_{22}(x)\bigr\},
\\
  \tilde{b}_{22}(x)
& =
  -\bigl\{\beta_{12}(x)-\beta_{21}(x)\bigr\}+i\bigl\{\beta_{11}(x)+\beta_{22}(x)\bigr\},
\\
  \tilde{c}_{11}(x)
& =
  \bigl\{\gamma_{11}(x)+\gamma_{22}(x)\bigr\}-i\bigl\{\gamma_{12}(x)-\gamma_{21}(x)\bigr\},
\\
  \tilde{c}_{12}(x)
& =
  \bigl\{\gamma_{11}(x)-\gamma_{22}(x)\bigr\}+i\bigl\{\gamma_{12}(x)+\gamma_{21}(x)\bigr\},
\\
  \tilde{c}_{21}(x)
& =
  \bigl\{\gamma_{11}(x)-\gamma_{22}(x)\bigr\}-i\bigl\{\gamma_{12}(x)+\gamma_{21}(x)\bigr\},
\\
  \tilde{c}_{22}(x)
& =
  \bigl\{\gamma_{11}(x)+\gamma_{22}(x)\bigr\}+i\bigl\{\gamma_{12}(x)-\gamma_{21}(x)\bigr\}.
\end{align*}
Hence, Theorem~\ref{theorem:main} implies that 
the initial value problem for \eqref{equation:301} is $L^2$-well-posed 
if and only if 
both \eqref{equation:condition11} and \eqref{equation:condition12} hold. 
This completes the proof. 
\end{proof}
We give the direct proof of the sufficiency of Theorem~\ref{theorem:main2}. 
We  begin with studying Fourier multipliers mapping 
real-valued functions to real-valued functions. 
Recall the requirements of the smooth function $\varphi_r(\xi)$: 
$0\leqslant\varphi_r(\xi)\leqslant1$, 
$\varphi_r(\xi)=1$ 
for 
$\lvert\xi\rvert\geqslant r+1$, 
$\varphi_r(\xi)=0$ 
for 
$\lvert\xi\rvert\leqslant r$, 
and 
$\varphi_r(\xi)=\varphi(-\xi)$. 
Note that the last one $\varphi_r(\xi)=\varphi(-\xi)$ is crucial here.   
Let $l$ be a nonnegative integer. 
We consider a Fourier multiplier $p_l(D_x)$ 
whose symbol is $p_l(\xi)=\varphi_r(\xi)/(i\xi)^l$. 
We make use of the following properties of $p_l(D_x)$. 
\begin{lemma}
\label{theorem:lemma} 
Let $v\in\mathscr{S}(\mathbb{R})$. 
If $\operatorname{Im}v(x)=0$, then $\operatorname{Im}p_l(D_x)v(x)=0$. 
\end{lemma}
\begin{proof}[{\bf Proof}] 
Suppose that $v\in\mathscr{S}(\mathbb{R})$ and $\operatorname{Im}v(x)=0$. 
By using change of variable $\xi\mapsto\eta=-\xi$, we deduce that 
\begin{align*}
& p_l(D_x)v(x)-\overline{p_l(D_x)v(x)}
\\
  =
& \frac{1}{2\pi}
  \iint_{\mathbb{R}^2}
  e^{i(x-y)\xi}
  \frac{\varphi_r(\xi)}{(i\xi)^l}
  v(y)
  dyd\xi
  -
  \overline{
  \frac{1}{2\pi}
  \iint_{\mathbb{R}^2}
  e^{i(x-y)\xi}
  \frac{\varphi_r(\xi)}{(i\xi)^l}
  v(y)
  dyd\xi
  }
\\
  =
& \frac{1}{2\pi}
  \iint_{\mathbb{R}^2}
  e^{i(x-y)\xi}
  \frac{\varphi_r(\xi)}{(i\xi)^l}
  v(y)
  dyd\xi
  -
  \frac{1}{2\pi}
  \iint_{\mathbb{R}^2}
  e^{i(x-y)(-\xi)}
  \frac{\varphi_r(\xi)}{(-i\xi)^l}
  v(y)
  dyd\xi
\\
  =
& \frac{1}{2\pi}
  \iint_{\mathbb{R}^2}
  e^{i(x-y)\xi}
  \frac{\varphi_r(\xi)}{(i\xi)^l}
  v(y)
  dyd\xi
  -
  \frac{1}{2\pi}
  \iint_{\mathbb{R}^2}
  e^{i(x-y)\eta}
  \frac{\varphi_r(-\eta)}{(i\eta)^l}
  v(y)
  dyd\xi
\\
  =
& \frac{1}{2\pi}
  \iint_{\mathbb{R}^2}
  e^{i(x-y)\xi}
  \frac{\varphi_r(\xi)}{(i\xi)^l}
  v(y)
  dyd\xi
  -
  \frac{1}{2\pi}
  \iint_{\mathbb{R}^2}
  e^{i(x-y)\eta}
  \frac{\varphi_r(\eta)}{(i\eta)^l}
  v(y)
  dyd\xi
  =0.
\end{align*}
This completes the proof.  
\end{proof}
Finally we prove the sufficiency of 
both \eqref{equation:condition11} and \eqref{equation:condition12} directly. 
We denote by $L^2(\mathbb{T};\mathbb{R}^2)$ 
the set of all $\mathbb{R}^2$-valued square-integrable functions on $\mathbb{T}$, 
and by $C^\infty(\mathbb{T};\mathbb{R})$ 
the set of all real-valued smooth functions on $\mathbb{T}$.   
\begin{proof}[{\bf Direct Proof of the Sufficiency in Theorem~\ref{theorem:main2}}]
Suppose that both 
\eqref{equation:condition11} and \eqref{equation:condition12} hold. 
We introduce a gauge transform on 
$L^2(\mathbb{T};\mathbb{R}^2)$, and modify the operator $\mathscr{L}$ 
so that $\beta(x)$ and $\gamma(x)$ become 
skew-symmetric and symmetric respectively. 
For simpler computations, we construct the gauge transform 
by a product of three pseudodifferential operators of order zero. 
For this reason we split the proof into three steps. 
Here we explain more detail of this strategy. 
Let $X=[x_{jk}]$ be a $2\times2$ real matrix. 
We split $X$ into three parts: 
\begin{align*}
  X
& =
  \frac{\operatorname{tr}(X)}{2}
  I
  +
  \frac{1}{2}
  \{X+JXJ\}
  +
  \frac{1}{2}
  \{X-JXJ-\operatorname{tr}(X)I\}. 
\\
& =
  \frac{\operatorname{tr}(X)}{2}
  I
  +
  \frac{1}{2}
  \{JX-XJ\}J
  -
  \frac{\operatorname{tr}(JX)}{2}
  J
\\
& =
  \frac{x_{11}+x_{22}}{2}
  I
  +
  \frac{1}{2}
  \begin{bmatrix}
  x_{11}-x_{22} & x_{12}+x_{21} \\ x_{12}+x_{21} & -x_{11}+x_{22} 
  \end{bmatrix}
  +
  \frac{-x_{12}+x_{21}}{2}
  J.  
\end{align*}
Note that both 
$(\operatorname{tr}(X))I/2$ 
and 
$-operatorname{tr}(JX)J/2$ 
commute with $J$, and that 
\begin{align*}
  JX-XJ
& =
  J
  \cdot
  \frac{1}{2}
  \{X+JXJ\}
  -
  \frac{1}{2}
  \{X+JXJ\}
  \cdot
  J,
\\ 
  \frac{1}{2}
  \{X+{}^tX\}
& =
  \frac{\operatorname{tr}(X)}{2}
  I
  +
  \frac{1}{2}
  \{JX-XJ\}J,
\\
  \frac{1}{2}
  \{X-{}^tX\}
& =
  \frac{1}{2}
  \{X-JXJ-\operatorname{tr}(X)I\}
  =
  -
  \frac{\operatorname{tr}(JX)}{2}
  J. 
\end{align*}
We eliminate 
$$
\frac{1}{2}
\{\beta(x)+{}^t\beta(x)\}
=
\frac{\operatorname{tr}\bigl(\beta(x)\bigr)}{2}
I
+
\frac{1}{2}
\bigl\{J\beta(x)-\beta(x)J\bigr\}
J
$$
from $\mathscr{L}$ in first two steps. 
In the third step we eliminate 
$$
\frac{1}{2}
\bigl\{\gamma(x)-{}^t\gamma(x)J\bigr\}
=
-
\frac{\operatorname{tr}\bigl(J\gamma(x)\bigr)}{2}
J
$$
from $\mathscr{L}$. 
\vspace{11pt}\\
{\bf Step 1}: 
We eliminate $\operatorname{tr}\bigl(\beta(x)\bigr)I/2$. 
Set 
$$
\Psi_4(x)
=
\int_0^x
\operatorname{tr}\bigl(\beta(y)\bigr)
dy. 
$$
We deduce that $\Psi_4 \in C^\infty(\mathbb{T};\mathbb{R})$ 
since $\beta_{jk}$ is a real-valued smooth functions on $\mathbb{T}$ 
and \eqref{equation:condition11}. 
Set 
$$
\sigma\bigl(\tilde{\Lambda}_4\bigr)(x,\xi)
=
\frac{1}{8}
\Psi_4(x)
\frac{\varphi_r(\xi)}{i\xi}
J,
\quad
\Lambda_4=I-\tilde{\Lambda}_4. 
$$
It follows that 
$\sigma\bigl(\tilde{\Lambda}_4\bigr)(x,\xi) \in S^{-1}\bigl(\mathbb{T};M(2)\bigr)$, 
$\sigma\bigl(\Lambda_4\bigr)(x,\xi) \in S^{0}\bigl(\mathbb{T};M(2)\bigr)$ 
and  
$\lVert\tilde{\Lambda}_4\rVert=\mathcal{O}(1/r)$.  
If we take a sufficiently large $r>0$, 
then $\Lambda_4$ is invertible on $\mathscr{L}\bigl(L^2(\mathbb{T};\mathbb{R}^2)\bigr)$, 
and the inverse satisfies 
$$
\Lambda_4^{-1}
=
I
+
\tilde{\Lambda}_4
+
\tilde{\Lambda}_4^2
+
\tilde{\Lambda}_4^3
+
\tilde{\Lambda}_4^4\Lambda_4^{-1}. 
$$
Set $\mathscr{L}_4=\Lambda_4\mathscr{L}\Lambda_4^{-1}$. 
Then $\Lambda_4\mathscr{L}=\mathscr{L}_4\Lambda_4$. 
We compute $\mathscr{L}_4$ in detail. 
Since $\tilde{\Lambda}_4$ is a pseudodifferential operator of order $-1$, 
we deduce that 
\begin{align}
  \mathscr{L}_4
& =
  \Lambda_4
  \left\{
  I
  \frac{\partial}{\partial t}
  +
  J
  \frac{\partial^4}{\partial x^4}
  +
  \beta(x)
  \frac{\partial^2}{\partial x^2}
  +
  \gamma(x)
  \frac{\partial}{\partial x}
  \right\}
  \Lambda_4^{-1}
\nonumber
\\
& \equiv
  I
  \frac{\partial}{\partial t}
  +
  \gamma(x)
  \frac{\partial}{\partial x}
\nonumber
\\
& +
  \bigl(
  I
  -
  \tilde{\Lambda}_4
  \bigr)
  \left\{
  J
  \frac{\partial^4}{\partial x^4}
  +
  \beta(x)
  \frac{\partial^2}{\partial x^2}
  \right\}
  \bigl(
  I
  +
  \tilde{\Lambda}_4
  +
  \tilde{\Lambda}_4^2
  +
  \tilde{\Lambda}_4^3
  +
  \tilde{\Lambda}_4^4\Lambda_4^{-1}
  \bigr)
\nonumber
\\
& \equiv
  I
  \frac{\partial}{\partial t}
  +
  J
  \frac{\partial^4}{\partial x^4}
  +
  \beta(x)
  \frac{\partial^2}{\partial x^2}
  +
  \gamma(x)
  \frac{\partial}{\partial x}
\nonumber
\\
& +
  \left(
  J
  \frac{\partial^4}{\partial x^4}
  \tilde{\Lambda}_4
  -
  \tilde{\Lambda}_4
  J
  \frac{\partial^4}{\partial x^4}
  \right)
  \bigl(
  I
  +
  \tilde{\Lambda}_4
  +
  \tilde{\Lambda}_4^2
  \bigr)
  +
  \left(
  \beta(x)
  \frac{\partial^2}{\partial x^2}
  \tilde{\Lambda}_4
  -
  \tilde{\Lambda}_4
  \beta(x)
  \frac{\partial^2}{\partial x^2}
  \right)
\label{equation:l4}
\end{align}
modulo $\mathscr{L}\bigl(L^2(\mathbb{T};\mathbb{R}^2)\bigr)$. 
We see the last two terms in detail. 
Since $J^2=-I$, we deduce that 
\begin{align}
& \sigma
  \left(
  J
  \frac{\partial^4}{\partial x^4}
  \tilde{\Lambda}_4
  -
  \tilde{\Lambda}_4
  J
  \frac{\partial^4}{\partial x^4}
  \right)
  (x,\xi)
\nonumber
\\
  \equiv
& \sum_{k=0}^2
  \frac{(-i)^k}{k!}
  \biggl[
  \left\{
  \frac{\partial^k}{\partial \xi^k}
  J(i\xi)^4
  \right\}
  \left\{
  \frac{\partial^k}{\partial x^k}
  \frac{1}{8}
  \Psi_4(x)
  \frac{\varphi_r(\xi)}{i\xi}
  J
  \right\}
\nonumber
\\
& \qquad\qquad\qquad
  -
  \left\{
  \frac{\partial^k}{\partial \xi^k}
  \frac{1}{8}
  \Psi_4(x)
  \frac{\varphi_r(\xi)}{i\xi}
  J
  \right\}
  \left\{
  \frac{\partial^k}{\partial x^k}
  J(i\xi)^4
  \right\}
  \biggr]
\nonumber
\\
  =
& \sum_{k=0}^2
  \frac{(-i)^k}{k!}
  \biggl[
  -
  \left\{
  \frac{\partial^k}{\partial \xi^k}
  (i\xi)^4
  \right\}
  \left\{
  \frac{\partial^k}{\partial x^k}
  \frac{1}{8}
  \Psi_4(x)
  \frac{\varphi_r(\xi)}{i\xi}
  \right\}
\nonumber
\\
& \qquad\qquad\qquad
  +
  \left\{
  \frac{\partial^k}{\partial \xi^k}
  \frac{1}{8}
  \Psi_4(x)
  \frac{\varphi_r(\xi)}{i\xi}
  \right\}
  \left\{
  \frac{\partial^k}{\partial x^k}
  (i\xi)^4
  \right\}
  \biggr]
  I
\nonumber
\\
  =
& -
  \sum_{k=1}^2
  \frac{(-i)^k}{k!}
  \left\{
  \frac{\partial^k}{\partial \xi^k}
  (i\xi)^4
  \right\}
  \left\{
  \frac{\partial^k}{\partial x^k}
  \frac{1}{8}
  \Psi_4(x)
  \frac{\varphi_r(\xi)}{i\xi}
  \right\}
  I
\nonumber
\\
  =
& -
  \frac{1}{2}
  \Psi_4^\prime(x)
  (i\xi)^2
  I
  -
  \frac{3}{4}
  \Psi_4^{\prime\prime}(x)
  (i\xi)
  I,
\label{equation:351}
\\
& \sigma
  \left(
  \left\{
  J
  \frac{\partial^4}{\partial x^4}
  \tilde{\Lambda}_4
  -
  \tilde{\Lambda}_4
  J
  \frac{\partial^4}{\partial x^4}
  \right\}
  \bigl(
  \tilde{\Lambda}_4
  +
  \tilde{\Lambda}_4^2
  \bigr)
  \right)
  (x,\xi)
\nonumber
\\
  \equiv
& \left\{
  -
  \frac{1}{2}
  \Psi_4^\prime(x)
  (i\xi)^2
  \right\}
  \left\{
  \frac{1}{8}
  \Psi_4(x)
  \frac{\varphi_r(\xi)}{i\xi}
  J
  \right\}
\nonumber
\\
  =
& -
  \frac{1}{32}
  \bigl\{\Psi_4(x)^2\bigr\}_x
  (i\xi)
  J,
\label{equation:352}
\\
& \sigma
  \left(
  \beta(x)
  \frac{\partial^2}{\partial x^2}
  \tilde{\Lambda}_4
  -
  \tilde{\Lambda}_4
  \beta(x)
  \frac{\partial^2}{\partial x^2}
  \right)
  (x,\xi)
\nonumber
\\
  \equiv
& \beta(x)
  (i\xi)^2
  \cdot
  \frac{1}{8}
  \Psi_4(x)
  \frac{\varphi_r(\xi)}{i\xi}
  J
  -
  \frac{1}{8}
  \Psi_4(x)
  \frac{\varphi_r(\xi)}{i\xi}
  J
  \cdot
  \beta(x)
  (i\xi)^2
\nonumber
\\
  \equiv
& \frac{1}{8}
  \Psi_4(x)
  \bigl\{\beta(x)J-J\beta(x)\bigr\}
  (i\xi),
\label{equation:353}
\end{align}
modulo $ S^0(\mathbb{T};M(2))$. 
Substituting 
\eqref{equation:351}, 
\eqref{equation:352} 
and 
\eqref{equation:353} 
into 
\eqref{equation:l4}, 
we obtain 
\begin{equation}
\mathscr{L}_4
\equiv
I
\frac{\partial}{\partial t}
+
J
\frac{\partial^4}{\partial x^4}
+
\beta_4(x)
\frac{\partial^2}{\partial x^2}
+
\gamma_4(x)
\frac{\partial}{\partial x}
\qquad\text{mod}\quad
\mathscr{L}\bigl(L^2(\mathbb{T};\mathbb{R}^2)\bigr),
\label{equation:L4} 
\end{equation}
\begin{equation}
\beta_4(x)
=
\beta(x)
-
\frac{1}{2}
\Psi_4^\prime(x)
I
=
\frac{1}{2}
\begin{bmatrix}
\beta_{11}(x)-\beta_{22}(x) & 2\beta_{12}(x)
\\
2\beta_{21}(x) & -\beta_{11}(x)+\beta_{22}(x)
\end{bmatrix},
\label{equation:beta4}
\end{equation} 
\begin{equation}
\gamma_4(x)
=
\gamma(x)
-
\frac{3}{4}
\Psi_4^{\prime\prime}(x)
I
-
\frac{1}{32}
\bigl\{\Psi_4(x)^2\bigr\}_x
J
+
\frac{1}{8}
\Psi_4(x)
\bigl\{\beta(x)J-J\beta(x)\bigr\}.
\label{equation:gamma4}
\end{equation}
\vspace{11pt}\\
{\bf Step 2}: 
We eliminate $\bigl\{J\beta(x)-\beta(x)J\bigr\}J/2$. 
Set
$$
\sigma\bigl(\tilde{\Lambda}_5\bigr)(x,\xi)
=
\frac{1}{2}
\beta_4(x)
J
p_2(\xi)
=
\frac{1}{2}
\beta_4(x)
J
\frac{\varphi_r(\xi)}{(i\xi)^2}, 
\quad
\Lambda_5=I+\tilde{\Lambda}_5.  
$$
It follows that 
$\sigma\bigl(\tilde{\Lambda}_5\bigr)(x,\xi) \in S^{-2}\bigl(\mathbb{T};M(2)\bigr)$, 
$\sigma\bigl(\Lambda_5\bigr)(x,\xi) \in S^{0}\bigl(\mathbb{T};M(2)\bigr)$ 
and  
$\lVert\tilde{\Lambda}_5\rVert=\mathcal{O}(1/r^2)$.  
If we take a sufficiently large $r>0$, 
then $\Lambda_5$ is invertible on $\mathscr{L}\bigl(L^2(\mathbb{T};\mathbb{R}^2)\bigr)$, 
and the inverse satisfies 
$\Lambda_5^{-1}=I-\tilde{\Lambda}_5+\tilde{\Lambda}_5^2\Lambda_5^{-1}$. 
Set 
$\mathscr{L}_5=\Lambda_5\mathscr{L}_4\Lambda_5^{-1}=\Lambda_5\Lambda_4\mathscr{L}\Lambda_4^{-1}\Lambda_5^{-1}$. 
Then $\Lambda_5\mathscr{L}_4=\mathscr{L}_5\Lambda_5$. 
We compute $\mathscr{L}_5$ in detail. 
Since $\tilde{\Lambda}_5$ is a pseudodifferential operator of order $-2$, 
we deduce that 
\begin{align}
  \mathscr{L}_5
& =
  \Lambda_5
  \left\{
  I
  \frac{\partial}{\partial t}
  +
  J
  \frac{\partial^4}{\partial x^4}
  +
  \beta_4(x)
  \frac{\partial^2}{\partial x^2}
  +
  \gamma_4(x)
  \frac{\partial}{\partial x}
  \right\}
  \Lambda_5^{-1}
\nonumber
\\
& \equiv
  I
  \frac{\partial}{\partial t}
  +
  \beta_4(x)
  \frac{\partial^2}{\partial x^2}
  +
  \gamma_4(x)
  \frac{\partial}{\partial x}
  +
  \bigl(
  I
  +
  \tilde{\Lambda}_5
  \bigr)
  J
  \frac{\partial^4}{\partial x^4}
  \bigl(
  I
  -
  \tilde{\Lambda}_5
  \bigr)
\nonumber
\\
& \equiv
  I
  \frac{\partial}{\partial t}
  +
  J
  \frac{\partial^4}{\partial x^4}
  +
  \beta_4(x)
  \frac{\partial^2}{\partial x^2}
  +
  \gamma_4(x)
  \frac{\partial}{\partial x}
  +
  \left(
  \tilde{\Lambda}_5
  J
  \frac{\partial^4}{\partial x^4}  
  -
  J
  \frac{\partial^4}{\partial x^4}
  \tilde{\Lambda}_5
  \right)
\label{equation:l5}
\end{align}
modulo $\mathscr{L}\bigl(L^2(\mathbb{T};\mathbb{R}^2)\bigr)$. 
We see the last term in detail. 
Since $\beta_4(x)J^2=-\beta(x)$ and $J\beta_4(x)J={}^t\beta_4(x)$, we deduce that 
\begin{align}
& \sigma
  \left(
  \tilde{\Lambda}_5
  J
  \frac{\partial^4}{\partial x^4}  
  -
  J
  \frac{\partial^4}{\partial x^4}
  \tilde{\Lambda}_5
  \right)
  (x,\xi)
\nonumber
\\
  \equiv
& \sum_{k=0}^1
  \frac{(-i)^k}{k!}
  \biggl[
  \left\{
  \frac{\partial^k}{\partial \xi^k}
  \frac{1}{2}
  \beta_4(x)J
  p_2(\xi)
  \right\}
  \left\{
  \frac{\partial^k}{\partial x^k}
  J(i\xi)^4
  \right\}
\nonumber
\\
& \qquad\qquad\qquad
  -
  \left\{
  \frac{\partial^k}{\partial \xi^k}
  J(i\xi)^4
  \right\}
  \left\{
  \frac{\partial^k}{\partial x^k}
  \frac{1}{2}
  \beta_4(x)J
  p_2(\xi)
  \right\}
  \biggr] 
\nonumber
\\
  =
& -
  \frac{1}{2}
  \sum_{k=0}^1
  \frac{(-i)^k}{k!}
  \biggl[
  \left\{
  \beta_4(x)
  \frac{\partial^k}{\partial \xi^k}
  \frac{\varphi_r(\xi)}{(i\xi)^2}
  \right\}
  \left\{
  \frac{\partial^k}{\partial x^k}
  (i\xi)^4
  \right\}
\nonumber
\\
& \qquad\qquad\qquad\quad
  +
  \left\{
  \frac{\partial^k}{\partial \xi^k}
  (i\xi)^4
  \right\}
  \left\{
  \frac{\partial^k}{\partial x^k}
  {}^t\beta_4(x)
  \frac{\varphi_r(\xi)}{(i\xi)^2}
  \right\}
  \biggr] 
\nonumber
\\
  \equiv
& -
  \frac{1}{2}
  \bigl\{\beta_4(x)+{}^t\beta_4(x)\bigr\}
  (i\xi)^2
  -
  2
  \frac{\partial {}^t\beta_4}{\partial x}(x)
  (i\xi)
\qquad\text{mod}\quad
S^0\bigl(\mathbb{T};M(2)\bigr). 
\label{equation:360}
\end{align}
Substitute \eqref{equation:360} into \eqref{equation:l5}. 
By using \eqref{equation:beta4} and \eqref{equation:gamma4}, we deduce 
\begin{align}
  \mathscr{L}_5
& \equiv
    I
  \frac{\partial}{\partial t}
  +
  J
  \frac{\partial^4}{\partial x^4}
  +
  \frac{1}{2}\{\beta_4(x)-{}^t\beta_4(x)\}
  \frac{\partial^2}{\partial x^2}
  +
  \left\{
  \gamma_4(x)
  -
  2
  \frac{\partial {}^t\beta_4}{\partial x}(x)
  \right\}
  \frac{\partial}{\partial x}
\nonumber
\\
& =
  I
  \frac{\partial}{\partial t}
  +
  J
  \frac{\partial^4}{\partial x^4}
  -
  \frac{\beta_{12}(x)-\beta_{21}(x)}{2}
  J
  \frac{\partial^2}{\partial x^2}
  +
  \left\{
  \gamma_4(x)
  -
  2
  \frac{\partial {}^t\beta_4}{\partial x}(x)
  \right\}
  \frac{\partial}{\partial x}
\nonumber
\\
& =
  I
  \frac{\partial}{\partial t}
  +
  J
  \frac{\partial^4}{\partial x^4}
  -
  \frac{\partial}{\partial x}
  \frac{\operatorname{tr}\bigl(J\beta(x)\bigr)}{2}
  J
  \frac{\partial}{\partial x}
  +
  \gamma_5(x)
  \frac{\partial}{\partial x}
\label{equation:L5}
\end{align}
modulo $\mathscr{L}\bigl(L^2(\mathbb{T};\mathbb{R}^2)\bigr)$, 
where 
\begin{align}
  \gamma_5(x)
& =
  \gamma_4(x)
  -
  2
  \frac{\partial {}^t\beta_4}{\partial x}(x)
  +
  \frac{\bigl\{\operatorname{tr}\bigl(J\beta(x)\bigr)\bigr\}_x}{2}
  J
\nonumber
\\
& =
  \gamma(x)
  +
  \frac{1}{8}
  \Psi_4(x)
  \bigl\{\beta(x)J-J\beta(x)\bigr\}
  -
  \frac{3}{4}
  \Psi_4^{\prime\prime}(x)
  I
\nonumber
\\
& -
  \frac{1}{32}
  \bigl\{\Psi_4(x)^2\bigr\}_x
  J
  -
  2
  \frac{\partial {}^t\beta_4}{\partial x}(x)
  +
  \frac{\bigl\{\operatorname{tr}\bigl(J\beta(x)\bigr)\bigr\}_x}{2}
  J
\label{equation:361}
\end{align}
\vspace{11pt}\\
{\bf Step 3}: Skew-symmetric part of $\gamma_5(x)$. 
Set 
$$
\gamma_5^{\text{sym}}(x)
=
\frac{1}{2}
\bigl\{\gamma_5(x)+{}^t\gamma_5(x)\bigr\},
\quad
\frac{1}{2}
\bigl\{\gamma_5(x)-{}^t\gamma_5(x)\bigr\}
=
\mu(x)J 
$$
for short. 
Then $\gamma_5(x)=\gamma_5^{\text{sym}}(x)+\mu(x)J$. 
Simple computations with \eqref{equation:beta4} yield 
$$
\beta(x)J-J\beta(x)
=
\begin{bmatrix}
\beta_{12}(x)+\beta_{21}(x) 
&
-\beta_{11}(x)+\beta_{22}(x)
\\
-\beta_{11}(x)+\beta_{22}(x)
& 
-\beta_{12}(x)-\beta_{21}(x)  
\end{bmatrix},
\quad
$$
$$
\frac{1}{2}
\left\{
-
2
\frac{\partial {}^t\beta_4}{\partial x}(x)
+
2
\frac{\partial \beta_4}{\partial x}(x)
\right\}
=
\bigl\{
\beta_4(x)-{}^t\beta_4(x)
\bigr\}_x
=
-
\bigl\{
\operatorname{tr}\bigl(J\beta(x)\bigr)
\bigr\}_x
J. 
$$
Substituting these into \eqref{equation:361}, we have 
$$
\mu(x)
=
-
\frac{\operatorname{tr}\bigl(J\gamma(x)\bigr)}{2}
-
\left\{
\frac{\operatorname{tr}\bigl(J\beta(x)\bigr)}{2}
+
\frac{\Psi_4(x)^2}{32}
\right\}_x. 
$$
Set 
$$
\Psi_6(x)
=
-
\frac{1}{2}
\int_0^x
\operatorname{tr}\bigl(J\gamma(y)\bigr)
dy
-
\frac{\operatorname{tr}\bigl(J\beta(x)\bigr)}{2}
-
\frac{\Psi_4(x)^2}{32}. 
$$
We deduce that 
$\Psi_6 \in C^\infty(\mathbb{T};\mathbb{R})$ since \eqref{equation:condition12}, 
and that $\Psi_6^\prime(x)=\mu(x)$. 
\par
Here we introduce a system of pseudodifferential operators 
$\Lambda_6=I+\tilde{\Lambda}_6$ defined by 
$$
\sigma\bigl(\tilde{\Lambda}_6\bigr)(x,\xi)
=
\frac{1}{4}
\Psi_6(x)
I
p_2(\xi)
=
\frac{1}{4}
\Psi_6(x)
I
\frac{\varphi_r(\xi)}{(i\xi)^2}. 
$$
It follows that 
$\sigma\bigl(\tilde{\Lambda}_6\bigr)(x,\xi) \in S^{-2}\bigl(\mathbb{T};M(2)\bigr)$, 
$\sigma\bigl(\Lambda_6\bigr)(x,\xi) \in S^{0}\bigl(\mathbb{T};M(2)\bigr)$ 
and  
$\lVert\tilde{\Lambda}_6\rVert=\mathcal{O}(1/r^2)$.  
If we take a sufficiently large $r>0$, 
then $\Lambda_6$ is invertible on $\mathscr{L}\bigl(L^2(\mathbb{T};\mathbb{R}^2)\bigr)$, 
and the inverse satisfies 
$\Lambda_6^{-1}=I-\tilde{\Lambda}_6+\tilde{\Lambda}_6^2\Lambda_6^{-1}$. 
Set 
$\mathscr{L}_6=\Lambda_6\mathscr{L}_5\Lambda_6^{-1}=\Lambda_6\Lambda_5\Lambda_4\mathscr{L}\Lambda_4^{-1}\Lambda_5^{-1}\Lambda_6^{-1}$. 
Then $\Lambda_6\mathscr{L}_5=\mathscr{L}_6\Lambda_6$. 
We compute $\mathscr{L}_6$ in detail. 
Since $\tilde{\Lambda}_6$ is a pseudodifferential operator of order $-2$, 
we deduce that 
\begin{align}
  \mathscr{L}_6
& =
  \Lambda_6
  \left\{
  I
  \frac{\partial}{\partial t}
  +
  J
  \frac{\partial^4}{\partial x^4}
  -
  \frac{\partial}{\partial x}
  \frac{\operatorname{tr}\bigl(J\beta(x)\bigr)}{2}
  J
  \frac{\partial}{\partial x}
  +
  \bigl\{
  \gamma_5^{\text{sym}}(x)
  +
  \mu(x)J
  \bigr\}
  \frac{\partial}{\partial x}
  \right\}
  \Lambda_6^{-1}
\nonumber
\\
& \equiv
  I
  \frac{\partial}{\partial t}
  -
  \frac{\partial}{\partial x}
  \frac{\operatorname{tr}\bigl(J\beta(x)\bigr)}{2}
  J
  \frac{\partial}{\partial x}
\nonumber
\\
& +
  \bigl\{
  \gamma_5^{\text{sym}}(x)
  +
  \mu(x)J
  \bigr\}
  \frac{\partial}{\partial x}
  +
  \bigl(
  I
  +
  \tilde{\Lambda}_6
  \bigr)
  J
  \frac{\partial^4}{\partial x^4}
  \bigl(
  I
  -
  \tilde{\Lambda}_6
  \bigr)
\nonumber
\\
& \equiv
  I
  \frac{\partial}{\partial t}
  +
  J
  \frac{\partial^4}{\partial x^4}
  -
  \frac{\partial}{\partial x}
  \frac{\operatorname{tr}\bigl(J\beta(x)\bigr)}{2}
  J
  \frac{\partial}{\partial x}
\nonumber
\\
& +
  \bigl\{
  \gamma_5^{\text{sym}}(x)
  +
  \mu(x)J
  \bigr\}
  \frac{\partial}{\partial x}
  +
  \left(
  \tilde{\Lambda}_6
  J
  \frac{\partial^4}{\partial x^4}  
  -
  J
  \frac{\partial^4}{\partial x^4}
  \tilde{\Lambda}_6
  \right).
\label{equation:l6}
\end{align}
modulo $\mathscr{L}\bigl(L^2(\mathbb{T};\mathbb{R}^2)\bigr)$. 
We see the last term in detail. 
By using $\Psi_6^\prime(x)=\mu(x)$, we deduce that 
\begin{align}
& \sigma
  \left(
  \tilde{\Lambda}_6
  J
  \frac{\partial^4}{\partial x^4}  
  -
  J
  \frac{\partial^4}{\partial x^4}
  \tilde{\Lambda}_6
  \right)
  (x,\xi)
\nonumber
\\
  \equiv
& \sum_{k=0}^1
  \frac{(-i)^k}{k!}
  \biggl[
  \left\{
  \frac{\partial^k}{\partial \xi^k}
  \frac{1}{4}
  \Psi_6(x)
  I
  p_2(\xi)
  \right\}
  \left\{
  \frac{\partial^k}{\partial x^k}
  J(i\xi)^4
  \right\}
\nonumber
\\
& \qquad\qquad\qquad
  -
  \left\{
  \frac{\partial^k}{\partial \xi^k}
  J(i\xi)^4
  \right\}
  \left\{
  \frac{\partial^k}{\partial x^k}
  \frac{1}{4}
  \Psi_6(x)
  I
  p_2(\xi)
  \right\}
  \biggr] 
\nonumber
\\
  =
& \frac{1}{4}
  \sum_{k=0}^1
  \frac{(-i)^k}{k!}
  \biggl[
  \left\{
  \frac{\partial^k}{\partial \xi^k}
  \Psi_6(x)
  p_2(\xi)
  \right\}
  \left\{
  \frac{\partial^k}{\partial x^k}
  (i\xi)^4
  \right\}
\nonumber
\\
& \qquad\qquad\qquad
  -
  \left\{
  \frac{\partial^k}{\partial \xi^k}
  (i\xi)^4
  \right\}
  \left\{
  \frac{\partial^k}{\partial x^k}
  \Psi_6(x)
  p_2(\xi)
  \right\}
  \biggr] 
  J
\nonumber
\\
& =
  -
  \frac{1}{4}
  (-i)
  \left\{
  \frac{\partial}{\partial \xi}
  (i\xi)^4
  \right\}
  \left\{
  \Psi_6^\prime(x)
  \frac{\varphi_r(\xi)}{(i\xi)^2}
  \right\}
  J
\nonumber
\\
& \equiv
  -
  \mu(x)(i\xi)J
  \qquad\text{mod}\quad
  S^0\bigl(\mathbb{T};M(2)\bigr).
\label{equation:370}
\end{align}
Substituting \eqref{equation:370} into \eqref{equation:l6}, we obtain 
\begin{equation}
\mathscr{L}_6
\equiv
I
\frac{\partial}{\partial t}
+
J
\frac{\partial^4}{\partial x^4}
-
\frac{\partial}{\partial x}
\frac{\operatorname{tr}\bigl(J\beta(x)\bigr)}{2}
J
\frac{\partial}{\partial x}
+
\gamma_5^{\text{sym}}(x)
\frac{\partial}{\partial x}
\label{equation:L6}
\end{equation}
modulo $\mathscr{L}\bigl(L^2(\mathbb{T};\mathbb{R}^2)\bigr)$. 
It is easy to check that the initial value problem for $\mathscr{L}_6$ is $L^2$-well-posed. 
Hence \eqref{equation:pde2}-\eqref{equation:data2} is also $L^2$-well-posed 
since $\Lambda_6\Lambda_5\Lambda_4$ is automorphic on $L^2(\mathbb{T};\mathbb{R}^2)$. 
We omit the detail. 
This completes the proof.  
\end{proof}
Finally, we remark that the sufficiency of Theorem~\ref{theorem:main2} holds also 
in case that all the coefficients depend on time variable $t\in\mathbb{R}$. 
To state this precisely, we here introduce some function spaces. 
We denote the set of all bounded continuous 
$C^\infty(\mathbb{T};\mathbb{R})$-valued functions on $\mathbb{R}$ 
by $C_b\bigl(\mathbb{R};C^\infty(\mathbb{T};\mathbb{R})\bigr)$. 
Set 
$$
C_b^1\bigl(\mathbb{R};C^\infty(\mathbb{T};\mathbb{R})\bigr)
=
\Bigl\{
a(t,x) \in C^1\bigl(\mathbb{R};C^\infty(\mathbb{T};\mathbb{R})\bigr)
\ \Big\vert \ 
a(t,x), a_t(t,x) 
\in 
C_b\bigl(\mathbb{R};C^\infty(\mathbb{T};\mathbb{R})\bigr)
\Bigr\}.
$$ 
Consider the initial value problem of the form 
\begin{alignat}{2}
  \mathscr{P}\vec{w}
& =
  \vec{h}(t,x) 
& \quad
& \text{in}
  \quad
  \mathbb{R}\times\mathbb{T},
\label{equation:pde3}      
\\
  \vec{w}(0,x)
& =
  \vec{w}_0(x)
& \quad
& \text{in}
  \quad
  \mathbb{T},
\label{equation:data3}
\end{alignat}
where 
$$
\mathscr{P}
=
I
\frac{\partial}{\partial t}
+
J
\frac{\partial^4}{\partial x^4}
+
\tilde{\beta}(t,x)
\frac{\partial^2}{\partial x^2}
+
\tilde{\gamma}(t,x)
\frac{\partial}{\partial x},
$$
and all the components in 
$\tilde{\beta}(t,x)=[\tilde{\beta}_{jk}(t,x)]_{j,k=1,2}$ 
and 
$\tilde{\gamma}(t,x)=[\tilde{\gamma}_{jk}(t,x)]_{j,k=1,2}$ 
are supposed to belong to 
$C_b^1\bigl(\mathbb{R};C^\infty(\mathbb{T};\mathbb{R})\bigr)$. 
In the same way as the direct proof of the sufficiency in Theorem~\ref{theorem:main2}, 
we can prove the following. 
\begin{theorem}
\label{theorem:main3} 
If we assume that 
\begin{align}
  \operatorname{Im} 
  \int_0^{2\pi}
  \operatorname{tr}\bigl(\tilde{\beta}(t,x)\bigr)
  dx
& =
  0,
\label{equation:condition21}
\\
  \operatorname{Im} 
  \int_0^{2\pi}
  \operatorname{tr}\bigl(J\tilde{\gamma}(t,x)\bigr)
  dx
& =
  0
\label{equation:condition22} 
\end{align}
for any $t\in\mathbb{R}$, 
then the initial value problem 
\eqref{equation:pde3}-\eqref{equation:data3} 
is $L^2$-well-posed. 
\end{theorem}
Indeed, if $\Lambda_7(t)=I+\tilde{\Lambda}_7(t)$ is invertible and 
$\sigma\bigl(\tilde{\Lambda}_7(t)\bigr)$ is an 
$S^{-1}\bigl(\mathbb{T};M(2)\bigr)$-valued function of class $C_b^1$, 
then 
$$
\Lambda_7(t)
I
\frac{\partial}{\partial t}
\Lambda_7(t)^{-1}
=
I
\frac{\partial}{\partial t}
-
\frac{\partial \tilde{\Lambda}_7}{\partial t}
\Lambda_7(t)^{-1}
\equiv
I
\frac{\partial}{\partial t}
$$
modulo a class of all 
$\mathscr{L}\bigl(L^2(\mathbb{T};\mathbb{R}^2)\bigr)$-valued 
bounded continuous functions in $t\in\mathbb{R}$. 
This is the only difference between 
the sufficiency of Theorem~\ref{theorem:main2} and that of Theorem~\ref{theorem:main3}. 
The other parts of the proof of Theorem~\ref{theorem:main3} 
are exacctly the same as that of Theorem~\ref{theorem:main2}. 
We omit the detail. 
\section{Dispersive flows and moving frames}
\label{section:frame} 
We turn our attention to the dispersive flow \eqref{equation:onoderasystem}. 
In this section we derive a fourth-order dispersive system like \eqref{equation:pde2} 
from the equation of the dispersive flow \eqref{equation:onoderasystem}. 
We see that 
if the sectional curvature of the target Riemann surface $(N,\tilde{J},g)$ is constant, 
then the derived system satisfies the conditions 
\eqref{equation:condition21} and \eqref{equation:condition22}. 
We begin with some preliminaries for the derivation. 
\par
First, we present local expressions of covariant derivative along $u$. 
We denote the pullback bundle of TN by $u^{-1}TN$, 
and all the smooth sections of $u^{-1}TN$ by $\Gamma(u^{-1}TN)$. 
Let $u^1$ and $u^2$ be local coordinates of N, 
and let $\Gamma^\alpha_{\beta\gamma}(u)$ ($\alpha,\beta,\gamma=1,2$)  
be the Christoffel symbol of $(N,\tilde{J},g)$. 
For 
$$
Y
=
Y^1
\left(
\frac{\partial}{\partial u^1}
\right)_u
+
Y^2
\left(
\frac{\partial}{\partial u^2}
\right)_u
\in 
\Gamma(u^{-1}TN), 
$$
$\nabla_xY$ and $\nabla_tY$ are locally written as 
\begin{align*}
  \nabla_xY
& =
  \nabla_{u_x}^NY
  =
  \sum_{\alpha=1}^2
  \left\{
  \frac{\partial Y^\alpha}{\partial x}
  +
  \sum_{\beta,\gamma=1}^2
  \Gamma^\alpha_{\beta\gamma}(u)
  \frac{\partial u^\beta}{\partial x}
  Y^\gamma
  \right\}
  \left(
  \frac{\partial}{\partial u^\alpha}
  \right)_u,
\\
  \nabla_tY
& =
  \nabla_{u_t}^NY
  =
  \sum_{\alpha=1}^2
  \left\{
  \frac{\partial Y^\alpha}{\partial t}
  +
  \sum_{\beta,\gamma=1}^2
  \Gamma^\alpha_{\beta\gamma}(u)
  \frac{\partial u^\beta}{\partial t}
  Y^\gamma
  \right\}
  \left(
  \frac{\partial}{\partial u^\alpha}
  \right)_u.  
\end{align*}
\par
Secondly, we introcude a moving frame along $u$ to describe sections of $u^{-1}TN$.  
For the sake of simplicity, we assume in addition that 
$u_x(t,0)\ne0$ for all $t\in\mathbb{R}$. 
Under this assumption, we set $e_0(t)=u_x(t,0)/g_{u(t,0)}(u_x,u_x)^{1/2}$ for short. 
This is the unit tangent vector of the closed curve 
$C(t)=\{u(t,x) \ \vert \ x\in\mathbb{T}\}$ at $u(t,0)$ for any fixed $t\in\mathbb{R}$. 
Let $e(t,x)$ be the parallel transport of $e_0(t)$ along $C(t)$. 
In other words, 
$e(t,x)$ is the unique solution to the initial value problem of 
the system of linear ordinary differential equations 
with an independent variable $x\in\mathbb{T}$ 
and a parameter $t\in\mathbb{R}$ of the form
$$
\nabla_xe(t,\cdot)=0 
\quad\text{in}\quad
\mathbb{T},
\qquad
e(t,0)=e_0(t). 
$$
Since $\nabla^N\tilde{J}=0$, $\tilde{J}e$ solves the initial value problem 
$$
\nabla_x\tilde{J}e(t,\cdot)=0 
\quad\text{in}\quad
\mathbb{T},
\qquad
\tilde{J}e(t,0)=\tilde{J}e_0(t),  
$$
and is the parallel transport of the unit normal vector 
$\tilde{J}e_0(t)$ of $C(t)$ at $u(t,0)$. 
It is easy to see that the pair of $e$ and $\tilde{J}e$ is a moving frame along $u$. 
We remark that $e(t,x)$ is not necessarily $2\pi$-periodic in $x$ 
since the coefficients of the lower order terms 
$\Gamma^\alpha_{\beta\gamma}(u)\partial u^\beta/\partial x$ 
do not necessarily have $2\pi$-periodic primirive. 
In other words, roughly speaking, 
$$
\int_0^{2\pi}
\Gamma^\alpha_{\beta\gamma}(u)
\frac{\partial u^\beta}{\partial x}
dx
$$ 
does not necessarily vanish. 
\par
Let $R$ be the Riemann curvature tensor of $(N,\tilde{J},g)$. 
The sectional curvature at $u \in N$ is denoted by $K(u)$.  
Here we summarize propeties of $R$ used in our computations below. 
\begin{lemma}
\label{theorem:curvature}
We have the following properties.
\begin{itemize}
\item[{\rm (i)}] 
$R(e,e)e=R(\tilde{J}e,\tilde{J}e)e=0$. 
\item[{\rm (ii)}] 
$\tilde{J}R(\cdot,\cdot)e=R(\cdot,\cdot)\tilde{J}e$. 
\item[{\rm (iii)}] 
$R(\tilde{J}e,e)e=-R(e,\tilde{J}e)e=K(u)\tilde{J}e$, 
$R(\tilde{J}e,e)\tilde{J}e=-R(e,\tilde{J}e)\tilde{J}e=-K(u)e$. 
\end{itemize}
\end{lemma}
\begin{proof}[{\bf Proof}] 
Let $X,Y,Z,W$ be vector fields on $N$. 
The claim (i) follows from a basic property of Riemann curvature tensor 
$R(X,Y)Z+R(Y,X)Z=0$. 
The claim (ii) follows from the definition of Riemann curvature tensor 
$$
R(X,Y)Z
=
\nabla^N_X\nabla^N_YZ
-
\nabla^N_Y\nabla^N_XZ
-
\nabla^N_{XY-YX}Z
$$
and the K\"ahler condition $\nabla^N\tilde{J}=0$. 
For the claim (iii), it suffices to show that $R(\tilde{J}e,e)e=K\tilde{J}e$. 
Since the pair of $e$ and $\tilde{J}e$ is an orthonormal basis of each tangent space $T_uN$, 
we have 
$$
R(\tilde{J}e,e)e
=
g\bigl(R(\tilde{J}e,e)e,e\bigr)
e
+
g\bigl(R(\tilde{J}e,e)e,\tilde{J}e\bigr)
\tilde{J}e. 
$$
On one hand, combining (i) and a basic property of Riemann curvature tensor 
$$
g\bigl(R(X,Y)Z,W\bigr)
=
g\bigl(R(Z,W)X,Y\bigr),  
$$
we have $g\bigl(R(\tilde{J}e,e)e,e\bigr)=0$. 
On the other hand, the definition of the sectional curvature at $u \in N$ is 
$$
K(u)
=
\frac{g_u\bigl(R(X,Y)Y,X\bigr)}{g_u(X,X)g_u(Y,Y)-g_u(X,Y)^2} 
\quad\text{for}\quad
X, Y \in T_uN\setminus\{0\},  
$$
and we have $g\bigl(R(\tilde{J}e,e)e,\tilde{J}e\bigr) = K(u)$. 
Hence we obtain $R(\tilde{J}e,e)e=K(u)\tilde{J}e$. 
\end{proof}
\par
Here we begin with the derivation of a system of partial differential equations. 
Let $l$ be an integer not smaller than four. 
Set $u_x=\xi e +\eta \tilde{J}e$ and $U=\nabla_x^lu_x=Ve+W\tilde{J}e$. 
In what follows we denote different functions of 
$u$, $u_x$, ..., $\nabla_x^lu$ by the same notation ``$\text{OK}$'', 
and we set $g(\cdot,\cdot)=g_{u(t,x)}(\cdot,\cdot)$ and 
$\tilde{J}=\tilde{J}\bigl(u(t,x)\bigr)$ for short. 
Apply $\nabla_x^{l+1}$ to \eqref{equation:onoderasystem}. 
Then we have 
\begin{align}
  \nabla_tU
& =
  a\tilde{J}\nabla_x^4U+\tilde{J}\nabla_x^2U
\nonumber
\\
& +
  \sum_{\alpha=0}^{l-1}
  \nabla_x^{l-1-\alpha}
  \bigl\{
  R(u_t,u_x)
  \nabla_x^\alpha
  u_x
  \bigr\}
\label{equation:401}
\\
& +
  b 
  \sum_{\alpha+\beta+\gamma=l+1}
  \frac{(l+1)!}{\alpha!\beta!\gamma!}
  g\bigl(\nabla_x^\beta u_x,\nabla_x^\gamma u_x\bigr)
  \tilde{J}\nabla_x^{\alpha+1}u_x
\label{equation:402}
\\
& +
  c
  \sum_{\alpha+\beta+\gamma=l+1}
  \frac{(l+1)!}{\alpha!\beta!\gamma!}
  g\bigl(\nabla_x^{\beta+1} u_x,\nabla_x^\gamma u_x\bigr)
  \tilde{J}\nabla_x^\alpha u_x. 
\label{equation:403}
\end{align}
We split each term in the above equation into main part and OK part. 
It follows from the definition of the covariant derivative $\nabla_t$ that 
\begin{equation}
\nabla_tU
=
V_te
+
W_t\tilde{J}e
+
V\nabla_te
+
W \tilde{J}\nabla_te
=
V_te
+
W_t\tilde{J}e
+ 
\ \text{OK}.
\label{equation:411} 
\end{equation}
Since $\nabla_xe=\nabla_x\tilde{J}e=0$, we have 
\begin{align}
  a\tilde{J}\nabla_x^4U
  +
  \tilde{J}\nabla_x^2U
& =
  a\tilde{J}\nabla_x^4(Ve+W\tilde{J}e)
  +
  \tilde{J}\nabla_x^2(Ve+W\tilde{J}e)
\nonumber
\\
& =
  -
  aW_{xxxx}e
  +
  aV_{xxxx}\tilde{J}e
  -
  W_{xx}e
  +
  V_{xx}\tilde{J}e.
\label{equation:412} 
\end{align}
\par
We compute \eqref{equation:401}. 
A simple computation yields 
\begin{align}
  \text{\eqref{equation:401}}
& =
  \sum_{\alpha=0}^{l-1}
  \nabla_x^{l-1-\alpha}
  \bigl\{
  R(a\tilde{J}\nabla_x^3u_x,u_x)
  \nabla_x^\alpha
  u_x
  \bigr\}
  +
  \text{OK}
\nonumber
\\
& =
  \sum_{\alpha=0}^1
  \nabla_x^{l-1-\alpha}
  \bigl\{
  R(a\tilde{J}\nabla_x^3u_x,u_x)
  \nabla_x^\alpha
  u_x
  \bigr\}
  +
  \text{OK}. 
\label{equation:420}
\end{align}
Set 
$$
\xi^{(\alpha)}(t,x)
=
\frac{\partial^\alpha \xi}{\partial x^\alpha}
(t,x),
\quad
\eta^{(\alpha)}(t,x)
=
\frac{\partial^\alpha \eta}{\partial x^\alpha}
(t,x),
\quad
\alpha=0,1,2,\dotsc
$$
for short. 
We applying Lemma~\ref{theorem:curvature} to the right hand side of \eqref{equation:420}. 
We deduce that 
\begin{align}
  R(a\tilde{J}\nabla_x^3u_x,u_x)
  \nabla_x^\alpha
  u_x
& = 
  R
  \bigl(
  a\tilde{J}
  (\xi^{(3)} e + \eta^{(3)} \tilde{J} e),
  \xi e + \eta \tilde{J} e
  \bigr) 
  (\xi^{(\alpha)} e + \eta^{(\alpha)} \tilde{J} e)
\nonumber
\\
& =
  a
  R
  \bigl(
  \xi^{(3)} \tilde{J} e - \eta^{(3)} e),
  \xi e + \eta \tilde{J} e
  \bigr) 
  (\xi^{(\alpha)} e + \eta^{(\alpha)} \tilde{J} e)
\nonumber
\\
& =
  a
  \{\xi\xi^{(3)}+\eta\eta^{(3)}\}
  R(Je,e)
  (\xi^{(\alpha)} e + \eta^{(\alpha)} \tilde{J} e)
\nonumber
\\
& =
  \bigl\{
  -
  aK(u)\xi\eta^{(\alpha)}\xi^{(3)}
  -
  aK(u)\eta\eta^{(\alpha)}\eta^{(3)}
  \bigr\}
  e
\nonumber
\\
& +
  \bigl\{
  aK(u)\xi\xi^{(\alpha)}\xi^{(3)}
  +
  aK(u)\eta\xi^{(\alpha)}\eta^{(3)}
  \bigr\}
  \tilde{J}e.
\label{equation:421}
\end{align}
Substitute 
\eqref{equation:421}, 
into 
\eqref{equation:420}.  
We obtain 
\begin{align}
  \text{\eqref{equation:401}}
& =  
  \{-aK(u)\xi\eta V_{xx} - aK(u)\eta^2 W_{xx}\}
  e
\nonumber
\\
& +
  \{aK(u)\xi^2 V_{xx} + aK(u)\xi\eta W_{xx}\}
  \tilde{J}e
\nonumber
\\
& +
  \bigl\{
  -
  a(l-1)
  \bigl(K(u)\xi\eta\bigr)_x
  V_x
  -
  a(l-1)
  \bigl(K(u)\eta^2\bigr)_x
  W_x
  \bigr\}
  e
\nonumber
\\
& +
  \bigl\{
  a(l-1)
  \bigl(K(u)\xi^2\bigr)_x
  V_x
  +
  a(l-1)
  \bigl(K(u)\xi\eta\bigr)_x
  W_x
  \bigr\}
  \tilde{J}e
\nonumber
\\
& +
  \{
  -
  aK(u)\xi\eta_xV_x
  -
  aK(u)\eta\eta_xW_x
  \}
  e
\nonumber
\\
& +
  \{
  aK(u)\xi\xi_xV_x
  +
  aK(u)\eta\xi_xW_x
  \}
  \tilde{J}e
  +
  \text{OK}.
\label{equation:424}
\end{align}
\par
We compute \eqref{equation:402}. Since $\nabla^Ng=0$, we deduce that 
\begin{align}
  \text{\eqref{equation:402}}
& =
  \ 
  \text{sum of cases of}\  
\nonumber
\\
& \qquad
  (\alpha,\beta,\gamma)
  =
  (l+1,0,0), 
  (l,1,0),
  (l,0,1),
  (0,l+1,0),
  (0,0,l+1)
\nonumber
\\
& + 
  \ \text{OK}
\nonumber
\\
& =
  bg(u_xu_x)\tilde{J}\nabla_x^2U
  +
  2(l+1)b
  g(\nabla_xu_x,u_x)\tilde{J}\nabla_xU
\nonumber
\\
& +
  2b
  g(\nabla_xU,u_x)
  \tilde{J}\nabla_xu_x
  +
  \text{OK}
\nonumber
\\
& =
  b
  \{\xi^2+\eta^2\}
  \tilde{J}\{V_{xx} e + W_{xx} \tilde{J}e\}
  +
  (l+1)b
  \{\xi^2+\eta^2\}_x
  \tilde{J}\{V_x e + W_x \tilde{J}e\}
\nonumber
\\
& +
  2b
  g\bigl(V_x e + W_x \tilde{J}e, \xi e + \eta \tilde{J}e\bigr)
  \tilde{J}\{\xi_x e + \eta_x \tilde{J}e\}
  +
  \text{OK}
\nonumber
\\
& =
  b
  \{\xi^2+\eta^2\}
  \{-W_{xx} e + V_{xx} \tilde{J}e\}
  +
  (l+1)b
  \{\xi^2+\eta^2\}_x
  \{-W_x e + V_x \tilde{J}e\}
\nonumber
\\
& +
  2b
  \{\xi V_x + \eta W_x\}
  \{-\eta_x e + \xi_x \tilde{J}e\}
  +
  \text{OK}
\nonumber
\\
& =
  b
  \{\xi^2+\eta^2\}
  \{-W_{xx} e + V_{xx} \tilde{J}e\}
\nonumber
\\
& +
  b
  \bigl[
  -
  2\xi\eta_x V_x
  -
  \{(l+1)\xi^2+(l+2)\eta^2\}_x W_x
  \bigr]
  e
\nonumber
\\
& +
  b
  \bigl[
  \{(l+2)\xi^2+(l+1)\eta^2\}_x V_x
  +
  2\xi_x\eta W_x
  \bigr]
  \tilde{J}e
  +
  \text{OK}.
\label{equation:425}  
\end{align}
\par
We compute \eqref{equation:403}. 
In the same way as \eqref{equation:425}, we deduce that 
\begin{align}
  \text{\eqref{equation:403}}
& =
  \ 
  \text{sum of cases of}\  
\nonumber
\\
& \qquad
  (\alpha,\beta,\gamma)
  =
  (0,l+1,0), 
  (1,l,0),
  (0,l,1),
  (l+1,0,0),
  (0,0,l+1)
\nonumber
\\
& + 
  \ \text{OK}
\nonumber
\\
& =
  cg(\nabla_x^2U,u_x)\tilde{J}u_x
\nonumber
\\
& +
  c(l+1)
  g(\nabla_xU,u_x)\tilde{J}\nabla_xu_x
  +
  c(l+1)
  g(\nabla_xU,\nabla_xu_x)\tilde{J}u_x
\nonumber
\\
& +
  cg(\nabla_xu_x,u_x)\tilde{J}\nabla_xU
  +
  cg(\nabla_xu_x,\nabla_xU)\tilde{J}u_x
  +
  \text{OK}
\nonumber
\\
& =
  c
  g\bigl(V_{xx} e + W_{xx} \tilde{J}e, \xi e + \eta \tilde{J}e\bigr)
  \tilde{J}\{\xi e + \eta \tilde{J}e\}
\nonumber
\\
& +
  c(l+1)
  g\bigl(V_x e + W_x \tilde{J}e, \xi e + \eta \tilde{J}e\bigr)
  \tilde{J}\{\xi_x e + \eta_x \tilde{J}e\}
\nonumber
\\
& +
  c(l+2)
  g\bigl(V_x e + W_x \tilde{J}e, \xi_x e + \eta_x \tilde{J}e\bigr)
  \tilde{J}\{\xi e + \eta \tilde{J}e\}
\nonumber
\\
& +
  \frac{c}{2}
  \{\xi^2+\eta^2\}_x
  \tilde{J}\{V_x e + W_x \tilde{J}e\}
  +
  \text{OK}
\nonumber
\\
& =
  c
  \{\xi V_{xx} + \eta W_{xx}\}
  \{- \eta e + \xi \tilde{J}e\}
\nonumber
\\
& +
  c(l+1)
  \{\xi V_x + \eta W_x\}
  \{- \eta_x e + \xi_x \tilde{J}e\}
\nonumber
\\
& +
  c(l+2)
  \{\xi_x V_x + \eta_x W_x\}
  \{-\eta e +\xi \tilde{J}e\}
\nonumber
\\
& +
  \frac{c}{2}
  \{\xi^2+\eta^2\}_x
  \{- W_x e + V_x \tilde{J}e\}
  +
  \text{OK}
\nonumber
\\
& =
  c
  \{-\xi\eta V_{xx} -\eta^2 W_{xx}\}e
  +
  c\{\xi^2 V_{xx} + \xi\eta W_{xx}\}\tilde{J}e
\nonumber
\\
& +
  c
  \left[
  -
  \{(l+1)(\xi\eta)_x+\xi_x\eta\}V_x
  -
  \left\{
  \frac{1}{2}
  \xi^2
  +
  (l+2)\eta^2
  \right\}_x
  W_x
  \right]
  e
\nonumber
\\
& +
  c
  \left[
  \left\{
  (l+2)\xi^2
  +
  \frac{1}{2}\eta^2
  \right\}_x
  V_x
  +
  \{(l+1)(\xi\eta)_x+\xi\eta_x\}
  W_x
  \right]
  \tilde{J}e
  +
  \text{OK}.
\label{equation:426}
\end{align}
\par
Combining 
\eqref{equation:411}, 
\eqref{equation:412}, 
\eqref{equation:424}, 
\eqref{equation:425} 
and  
\eqref{equation:426}, 
we obtain
\begin{equation}
\left\{
I
\frac{\partial}{\partial t}
-
a
J
\frac{\partial^4}{\partial x^4}
+
\hat{\beta}(t,x)
\frac{\partial}{\partial x^2}
+
\hat{\gamma}(t,x)
\frac{\partial}{\partial x}
\right\}
\begin{bmatrix}
V \\ W 
\end{bmatrix}
=
\ \text{OK},
\label{equation:derived} 
\end{equation}
where 
$$
\hat{\beta}(t,x)
=
\begin{bmatrix}
\hat{\beta}_{11}(t,x) & \hat{\beta}_{12}(t,x) 
\\ 
\hat{\beta}_{21}(t,x) & \hat{\beta}_{22}(t,x) 
\end{bmatrix},
\quad
\hat{\gamma}(t,x)
=
\begin{bmatrix}
\hat{\gamma}_{11}(t,x) & \hat{\gamma}_{12}(t,x) 
\\ 
\hat{\gamma}_{21}(t,x) & \hat{\gamma}_{22}(t,x)
\end{bmatrix}, 
$$
\begin{align*}
  \hat{\beta}_{11}
& =
  \{aK(u)+c\}\xi\eta,
\\ 
  \hat{\beta}_{12}
& =
  1+b\xi^2+\{aK(u)+b+c\}\eta^2,
\\ 
  \hat{\beta}_{21}
& =
  -1-\{aK(u)+b+c\}\xi^2-b\eta^2,
\\ 
  \hat{\beta}_{22}
& =
  -\{aK(u)+c\}\xi\eta,
\end{align*}
\begin{align*}
  \hat{\gamma}_{11}
& =
  \frac{\partial}{\partial x}
  \left[
  \left\{
  a(l-1)K(u)+c(l+2)
  \right\}
  \xi\eta
  \right]
  +
  \{
  aK(u)+2b-c
  \}
  \xi\eta_x,
\\ 
  \hat{\gamma}_{12}
& =
  \frac{\partial}{\partial x}
  \left[
  \left\{
  b(l+2)+\frac{c}{2}
  \right\}
  \xi^2
  +
  \left\{
  \frac{a}{2}(2l-1)K(u)
  +
  (b+c)(l+2)
  \right\}
  \eta^2
  \right]
\nonumber
\\
& -
  \frac{a}{2}
  \left\{
  \frac{\partial}{\partial x}
  K(u)
  \right\}
  \eta^2,
\\ 
  \hat{\gamma}_{21}
& =
  -
  \frac{\partial}{\partial x}
  \left[
  +
  \left\{
  \frac{a}{2}(2l-1)K(u)
  +
  (b+c)(l+2)
  \right\}
  \xi^2
  +
  \left\{
  b(l+2)+\frac{c}{2}
  \right\}
  \eta^2
  \right]
\nonumber
\\
& +
  \frac{a}{2}
  \left\{
  \frac{\partial}{\partial x}
  K(u)
  \right\}
  \xi^2,
\\ 
  \hat{\gamma}_{22}
& =
  -
  \frac{\partial}{\partial x}
  \left[
  \left\{
  a(l-1)K(u)+c(l+2)
  \right\}
  \xi\eta
  \right]
  -
  \{
  aK(u)+2b-c
  \}
  \xi_x\eta.
\end{align*}
It is easy to see that 
$\hat{\beta}_{11}(t,x)+\hat{\beta}_{22}(t,x)\equiv0$, 
$$
\hat{\gamma}_{12}(t,x)
-
\hat{\gamma}_{21}(t,x)
=
\frac{\partial}{\partial x}
\{H(u)g(u_x,u_x)\}
-
\frac{a}{2}
\left\{
\frac{\partial}{\partial x}
K(u)
\right\}
g(u_x,u_x),
$$
$$
H(u)
=
\frac{a}{2}
(2l-1)
K(u)
+
b(2l+3)
+
\frac{c}{2}(2l+5),
$$
$$
\int_0^{2\pi}
\{
\hat{\gamma}_{12}(t,x)
-
\hat{\gamma}_{21}(t,x)
\}
dx
=
-
\frac{a}{2}
\int_0^{2\pi}
\left\{
\frac{\partial}{\partial x}
K(u)
\right\}
g(u_x,u_x)
dx,
\quad
t\in\mathbb{R}. 
$$
\par
Unfortunately, $V(t,x)$ and $W(t,x)$ are not $2\pi$-periodic functions in $x$. 
We shall obtain a system for $2\pi$-periodic functions in $x$ by correction. 
We denote by $2\pi\theta(t)$ the correction angle for the closed curve $C(t)$, 
which is the angle formed by $e(t,0)$ and $e(t,2\pi)$ in $T_{u(t,0)}N$, 
and said to be the holonomy angle of $C(t)$ at $u(t,0) \in N$. 
If $C(t)$ is the boundary enclosing a contractive domain  $D(t)$, 
then $\theta(t)$ is given by 
$$
\theta(t)
=
\frac{1}{2\pi}
\int_{D(t)}
K(u)
du^1 \wedge du^2.
$$
See \cite[Section~7.3]{ST} for this.  
Set 
$$
\vec{Z}(t,x)
=
P\bigl(\theta(t)x\bigr)
\begin{bmatrix}
V(t,x) \\ W(t,x)
\end{bmatrix},
\qquad
P(s)
=
\begin{bmatrix}
\cos{s} & -\sin{s} \\ \sin{s} & \cos{s}
\end{bmatrix},
\quad
s\in\mathbb{R}. 
$$
Then $\vec{Z}$ is $2\pi$-periodic in $x$. 
The normalized angle $\theta(t)$ is determied by $u(t,x)$ and $u_x(t,x)$, 
and 
$\theta(t)$ is $C^1$ provided that 
$u,u_x$, 
$\nabla_xu_x$, 
$\nabla_x^2u_x$, 
$\nabla_x^3u_x,$ 
and 
$\nabla_x^4u_x$ 
are continuous. 
Note that 
$$
J=P\left(\frac{\pi}{2}\right),
\quad
\frac{\partial}{\partial x}
P\left(\theta(t)x\right)
=
\theta(t)
P\left(\theta(t)x+\frac{\pi}{2}\right)
=
\theta(t)
J
P\left(\theta(t)x\right)
=
\theta(t)
P\left(\theta(t)x\right)
J. 
$$
Set $P=P\bigl(\theta(t)x\bigr)$ for short. 
Multiply \eqref{equation:derived} by $P$ from the left. 
Then 
\begin{equation}
P
\left\{
I
\frac{\partial}{\partial t}
-
a
J
\frac{\partial^4}{\partial x^4}
+
\hat{\beta}(t,x)
\frac{\partial}{\partial x^2}
+
\hat{\gamma}(t,x)
\frac{\partial}{\partial x}
\right\}
{}^tP
\vec{Z}
=
P\vec{F}(t,x).
\label{equation:480} 
\end{equation}
We compute this in detail. 
Simple computations give 
\begin{align}
  P
  I\frac{\partial}{\partial t}
  {}^tP
& =
  I\frac{\partial}{\partial t}
  -
  \theta(t)J,
\label{equation:481}
\\
  -a
  P
  J\frac{\partial^4}{\partial x^4}
  {}^tP
& =
  -
  aJ\frac{\partial^4}{\partial x^4}
  -
  4a\theta(t)
  I
  \frac{\partial^3}{\partial x^3}
  +
  6a\theta(t)^2
  J
  \frac{\partial^2}{\partial x^2}
  +
  4a\theta(t)^3
  I
  \frac{\partial}{\partial x}
  -
  a\theta(t)^4
  J,
\label{equation:482}
\\
  P
  \hat{\beta}(t,x)
  \frac{\partial^2}{\partial x^2}
  {}^tP
& =
  P\hat{\beta}(t,x){}^tP
  \frac{\partial^2}{\partial x^2}
  -
  2\theta(t)
  P\hat{\beta}(t,x){}^tPJ
  \frac{\partial}{\partial x}
  -
  \theta(t)^2
  P\hat{\beta}(t,x){}^tP,
\label{equation:483}
\\
  P
  \hat{\gamma}(t,x)
  \frac{\partial}{\partial x}
  {}^tP
& =
  P\hat{\gamma}(t,x){}^tP
  \frac{\partial}{\partial x}
  -
  \theta(t)
  P\hat{\gamma}(t,x){}^tP
  J.
\label{equation:484}
\end{align}
Substitute 
\eqref{equation:481}, 
\eqref{equation:482}, 
\eqref{equation:483} 
and 
\eqref{equation:481} 
into 
\eqref{equation:480}. 
We obtain 
\begin{equation}
\left\{
I
\frac{\partial}{\partial t}
-
a
J
\frac{\partial^4}{\partial x^4}
-
a
\theta(t)
I
\frac{\partial^3}{\partial x^3}
+
\hat{\beta}_1(t,x)
\frac{\partial^2}{\partial x^2}
+
\hat{\gamma}_1(t,x)
\frac{\partial}{\partial x}
\right\}
\vec{Z}
=
\vec{F}_1(t,x),
\label{equation:corrected} 
\end{equation}
where 
\begin{align*}
  \hat{\beta}_1(t,x)
& =
  P\hat{\beta}(t,x){}^tP
  +
  6a\theta(t)^2J,
\\
  \hat{\gamma}_1(t,x)
& =
  P\hat{\gamma}_1(t,x){}^tP
  +
  4a\theta(t)^3I
  -
  2\theta(t)
  P\hat{\beta}(t,x){}^tPJ,
\end{align*}
and $\vec{F}_1(t,x)$ is a function of 
$u$, $u_x$, $\nabla_xu_x$, $\nabla_x^2u_x$, $\nabla_x^3u_x$ and $\nabla_x^4u_x$. 
It is easy to chack that  
$$
\operatorname{tr}\bigl(\hat{\beta}(t,x)\bigr)
=
\operatorname{tr}\bigl(\hat{\beta}_1(t,x)\bigr),
\quad
\operatorname{tr}\bigl(J\hat{\gamma}(t,x)\bigr)
=
\operatorname{tr}\bigl(J\hat{\gamma}_1(t,x)\bigr).
$$
We see \eqref{equation:corrected} 
as a system of partial differential equations for $\vec{Z}$. 
The third order term in \eqref{equation:corrected} 
has no essential influence on the well-posedness of the initial value problem. 
Theorem~\ref{theorem:main3} shows that 
if $K(u)$ is constant, 
then the initial value problem for \eqref{equation:corrected} is $L^2$-well-posed. 
In other words, if the sectional curvature of the target Riemann surface is constant,  
then the initial value problem for \eqref{equation:onoderasystem} is made to be solvable. 
\vspace{11pt}
\\
{\bf Acknowledgements}. 
The author would like to thank Eiji Onodera 
for invaluable comments and helpful information on the holomomy. 

\end{document}